\documentclass[
aip,
amsmath, amssymb,
reprint, onecolumn 
]{revtex4-1}

\usepackage{graphicx}
\usepackage{dcolumn}
\usepackage{bm}

\usepackage[utf8]{inputenc}
\usepackage[T1]{fontenc}
\usepackage{mathptmx}
\usepackage{etoolbox}

\usepackage{hyperref}

\usepackage{amsmath,amsthm,amssymb,amsfonts,amscd}
\usepackage{enumerate}
\usepackage[usenames]{color}
\usepackage{adjustbox}
\usepackage{euscript,times}
\usepackage{cleveref}
\usepackage{esint}
\usepackage{version}
\excludeversion{hide}

\newcommand{\bdm}{\begin{displaymath}}
\newcommand{\edm}{\end{displaymath}}
\newcommand{\define}{:=}

\newcommand{\dd}[1]{\, {\mathrm d}{#1}}
\newcommand{\dif}[2]{{\frac{\mathrm{d}{#1}}{\mathrm{d}{#2}}}}
\newcommand{\difh}[3]{{\frac{\mathrm{d}^{#1}{#2}}{\mathrm{d}\,{#3}^{#1}}}}

\renewcommand{\(}{\left(}
\renewcommand{\)}{\right)}

\newcommand{\Rf}{{\mathbb R}}

\newcommand{\im}{{\mathrm i}}

\renewcommand{\oint}{\ointctrclockwise}

\newcommand{\rank}{\operatorname{rank}}

\newcommand{\tr}{\operatorname{tr}}

\makeatletter
\def\@email#1#2{%
 \endgroup
 \patchcmd{\titleblock@produce}
  {\frontmatter@RRAPformat}
  {\frontmatter@RRAPformat{\produce@RRAP{*#1\href{mailto:#2}{#2}}}\frontmatter@RRAPformat}
  {}{}
}%
\makeatother
\begin{document}

\preprint{AIP/123-QED}

\title[Nonlocal Integral Operators and SLP]{On a Nonlocal Integral Operator Commuting with the Laplacian and the Sturm-Liouville Problem: Low Rank Perturbations of the Operator}
\author{Lotfi Hermi}
\email{lhermi@fiu.edu.}
\homepage{https://faculty.fiu.edu/~lhermi.}
\affiliation{ 
		Department of Mathematics and Statistics,\\
		Florida International University,\\
		Miami, FL, USA	
}%

\author{Naoki Saito}%
\email{saito@math.ucdavis.edu.}
 \homepage{https://www.math.ucdavis.edu/~saito/.}
\affiliation{ 
Department of Mathematics, \\
University of California, One Shields Avenue, \\
Davis, CA 95616 USA}

\date{\today}

\begin{abstract}
		We reformulate all general real coupled self-adjoint boundary value problems as integral operators and show that they are all finite rank perturbations of the free space Green's function on the real line. This free space Green's function corresponds to the nonlocal boundary value problem proposed earlier by Saito [N.~Saito, Appl.~Comput.~Harmonic Anal., \textbf{25}, 68--97 (2008)].
        We prove these perturbations to be polynomials of rank up to 4. They encapsulate in a fundamental way the corresponding boundary conditions.
\end{abstract}

\maketitle


\newtheorem{theorem}{Theorem}[section]
\newtheorem{lemma}[theorem]{Lemma}
\newtheorem{corollary}[theorem]{Corollary}
\newtheorem{definition}[theorem]{Definition}
\newtheorem{proposition}[theorem]{Proposition}
\newtheorem{remark}[theorem]{Remark}
\newtheorem{example}[theorem]{Example}
\allowdisplaybreaks

\crefname{lemma}{lemma}{lemmas}
\Crefname{lemma}{Lemma}{Lemmas}
\crefname{theorem}{theorem}{theorems}
\Crefname{theorem}{Theorem}{Theorems}
\crefname{remark}{remark}{remarks}
\Crefname{remark}{Remark}{Remarks}
\crefname{example}{example}{examples}

\section{Introduction}
\label{sec:intro}
In this article, we continue the program initiated in Ref.~\onlinecite{SAITO-LAPEIG-ACHA}, and revisited in Ref.~\onlinecite{HS-ACHA18}. 
In Ref.~\onlinecite{SAITO-LAPEIG-ACHA}, Saito proposed the use the eigenfunctions of the integral operator commuting with the Laplace operator for analyzing functions and data defined on a general shape domain $\Omega \subset \Rf^d$. Computing eigenfunctions of the integral operator is simpler, more stable, and much faster with mordern fast algorithms such as the Fast Multipole Method~\cite{GREENGARD-ROKHLIN,Xue2007} than directly solving the corresponding Laplacian eigenvalue problem (via the Helmholtz equation). This formulation imposes an interesting \textit{nonlocal boundary condition in the Helmholtz equation formulation}, and we have been interested in investigating the nature of this boundary condition. Restricting the domain $\Omega$ to intervals in $\Rf$, but extending the kernel function from that of Saito's original proposal, we continued our analysis of the integral operators and their spectral properties~\cite{HS-ACHA18}. 

This article addresses an old question of Gel'fand \& Levitan~\cite{GelfandLevitan53}, namely the nature of the spectrum under perturbation. We treat this question in an integral operator setting. This formulation is expressed in terms of the Green's function, or fundamental solution, that corresponds to an elliptic operator. Green's functions are an important tool of Quantum Field Theory, Electrodynamics, Seismology, and Partial Differential Equations which found new applications in Machine Learning. In Quantum Field Theory, the Green's function is called the \emph{propagator} \cite{Peskin-Schroeder1995} or \emph{two-point correlation function} \cite{Griffiths-Schroeter2018}. It encodes the probability of measuring a field at a given point given its source at another. In Seismology, the Green's function plays a fundamental role in the solution of elastodynamic systems~\cite{AKI-RICHARDS}.
Integral operators are prominent in Machine Learning  \cite{carreira, chui-wang, rosasco-belkin-devito, luxburg, vonLux08,vonLux04, Yao07}. They are ubiquitous in spectral clustering algorithms, kernel methods, and many manifold learning algorithms. Understanding their spectra is fundamental for various applications \cite{lee, LIEU-SAITO-ENSEMBLE, maaten}. Perturbation techniques are traditionally used to obtain closed forms of such kernels \cite{KATO}. In this article we show that
they can also reveal information about the boundary conditions more explicitly.

We focus on the one-dimensional setting of \emph{a nonlocal boundary value problem defined by a symmetric kernel on a finite interval}. The problems we treat have roots in the work of Marcel Riesz \cite{Riesz38,Riesz49}, and in the work of Schoenberg~\cite{Schoenberg1,Schoenberg2,Schoenberg3} in the discrete setting. A general formulation is given in Hellwig \cite[Chap.~4]{Hellwig}. We offer a unified alternative following Kato \cite{KATO} and more recent works by Gesztesy and his collaborators \cite{GK-JFA19,AGHKLT19,AGHKLT18,GK-survey19,GK-QAM19}. The spectral framework proposed in this article provides an efficient and systematic way of obtaining the Green's function, especially in the case where $\lambda=0$ is one of the eigenvalues (compare, e.g., with the Neumann boundary condition calculations in Porter \& Stirling \cite{PorterStirling}), or the more recent  Fucci \emph{et.\ al} \cite{FGKLNS2021,FGKNS2021} for a more general setting of the Kre\u{\i}n-von-Neumann eigenvalue problem. Our work aligns with Gesztesy \& Kirsten \cite{GK-survey19,GK-QAM19}. Our framework explains the complications in the expression of the Green's function in terms of Riesz projections.

	
In its simplest form, our operator of interest is a \emph{nonlocal} operator defined by 
\begin{equation} \label{eq:Saito-Operator} 
\mathcal{K} f \define -\frac{1}{2} \, \int_0^1 |x-y| \, f(y) \dd{y}.
\end{equation}
We are interested in the spectral properties of this operator. The eigenvalue problem associated with \cref{eq:Saito-Operator}, 
i.e., $ \mathcal{K} u = (1/\lambda) u$, corresponds to the following \emph{Sturm-Liouville eigenvalue problem:}
\begin{align}
\label{eq:Saito-Operator2} 
u''+ \lambda u &=0 \quad \text{in $(0,1)$}, \notag \\
u'(0) &= - u(0) - u(1), \notag \\
u'(1) &= u(0)+u(1) .
\end{align} 
We are also interested in the spectral properties of the general problem
\begin{align} \label{eq:general} 
u''+ \lambda u &= 0 \quad \text{in $(0,1)$}, \notag \\
u'(0) &= \alpha u(0) + \beta u(1) , \notag \\
u'(1) &= \gamma u(0)+ \delta u(1) ,
\end{align} 
where $\gamma=-\beta$ ($\neq 0$) to guarantee the self-adjointness of the operator 
(see Ref.~\onlinecite{FOLLAND-FOURIER}, especially Sec. 3.5, 3.6). 
We let \begin{equation*}
C \define \left( {\begin{array}{cc}
	\alpha & \beta \\
	-\beta & \delta 
	\end{array} }\right) \in \mathbb{R}^{2 \times 2} .
\end{equation*}
We observe that this boundary condition (BC) corresponds to the \emph{General Self-Adjoint Real-Coupled (``GSARC'' for short) BC} introduced in Refs.~\onlinecite{BEZ96,KZ96a, KZ96b} (see also Refs.~\onlinecite{EKWZ99,ZettlBook2005}). For a given $2\times2$ matrix $C\in SL(2,\mathbb{R})$, i.e., the Special Linear Group acting on $\mathbb{R}$, 
the GSARC problem is also given by
\begin{equation} \label{eq:kan}
\begin{pmatrix}
u(1)\\
u'(1)
\end{pmatrix}= B \, \begin{pmatrix}
u(0)\\
u'(0)
\end{pmatrix}
\end{equation}
Note that the restriction $\det B=1$ is imposed in the literature to simplify the problem. To see that \cref{eq:kan} follows from \cref{eq:general}, we calculate the matrix $B$, explicitly to obtain
\begin{equation*}
B = \left( {\begin{array}{cc}
	-\frac{\alpha}{\beta} & \frac{1}{\beta} \\
	-\frac{\beta^2 + \alpha \delta}{\beta} & \frac{\delta}{\beta} \\
	\end{array} }\right)
\end{equation*}
\begin{remark} The matrix $B$ is sometimes written in terms of the sines and cosines to reflect the self-adjointness; see, e.g., the works of F.~Gesztesy and his collaborators~\cite{FGKNS2021,FGKLNS2021,GK-JFA19}. To highlight this angular dependence, we note that one can indeed recast \cref{eq:Saito-Operator2} in the form \cref{eq:general}, using the KAN decomposition of $B$, viz.,
	\begin{equation} \label{eq:kan2}
	B = K A N
	\end{equation}
	with $K= \left( {\begin{array}{cc}
		\cos\theta & -\sin\theta \\
		\sin\theta & \cos\theta \\
		\end{array} }\right)$, 	$A= \left( {\begin{array}{cc}
		r & 0 \\
		0 & 1/r \\
		\end{array} }\right)$, and $N= \left( {\begin{array}{cc}
		1 & n \\
		0 & 1 \\
		\end{array} }\right)$, 	
	where $r>0$, and $\theta, n\in \mathbb{R}$. (This is a unique decomposition called ``the Iwasawa decomposition'' in the Lie theory literature; it is a special case of QR factorization; see \cite{delaCruzReyes21}). The correspondence between \cref{eq:kan2} and \cref{eq:general} is explicitly given by
	\[\alpha =\frac{\cos\theta + n r^2 \sin\theta}{n r^2 \cos\theta- \sin\theta}, \quad \beta=\frac{-r}{n r^2 \cos\theta- \sin\theta}, \text{ and } \quad \delta=\frac{r^2 \cos\theta}{n r^2 \cos\theta- \sin\theta}.\] 
	\qed
\end{remark}

\begin{remark} \label{rk I.2} When $\beta=-\gamma=0$ in \cref{eq:general} the problem corresponds to what can be termed as the General Self-Adjoint Real-Separated (``GSARS'' for short) BC, detailed in the work of Folland \cite{FOLLAND-FOURIER} (Section 3.5, pp.~86-95). In the series of articles of Zettl and his co-authors \cite{KZ96b, KZ96a,KWZ2008} on the geometric structures of spaces of boundary conditions, it sometimes takes the form
\begin{align} \label{eq:GSARS} 
u'(0) \sin \theta_0 + u(0) \cos \theta_0 &= 0 , \notag \\
u'(1) \sin \theta_1 + u(1) \cos \theta_1 &= 0,
\end{align} 
where $\theta_0, \theta_1 \in [0, \pi)$. These correspond to $\alpha=-\cot \theta_0$ and $\delta=-\cot \theta_1$ in our notation. We shall discuss these particular GSARS cases in this article.
A general framework for spaces of boundary conditions for `separated' and `coupled' self-adjoint BCs for higher order Sturm-Liouville problems can be found in the literature. \cite{Atkinson1964,WSZ2008,WSZ2011}.
\end{remark}


Recent activity \cite{AGLMS2017,AGMST2013,AGMST2010}, following \cite{AlonsoSimon80}, has focused on minimal extension of operators and their properties. We adopt a more classical point of view, \cite{Lidskii59, GohbergKrein1969} focusing on \emph{lower rank perturbation} of \cref{eq:Saito-Operator}. We will show that in the most general setting, all known examples of the Sturm-Liouville problem are polynomial perturbations of \cref{eq:Saito-Operator} by operators of rank at most 4. We believe the result to be new. 
This article proposes:	
\begin{itemize}
	\item A unified view of all known boundary value problems (BVPs) recasting them in terms of the integral operator formulation (e.g., Green's functions);
	\item All of these BVPs are equivalent --- up to the finite dimensional perturbation --- to the base nonlocal integral operator \cref{eq:Saito-Operator}. The BCs for all these problems are indeed encoded in the perturbation in the following sense: The nonlocal operator \cref{eq:Saito-Operator} corresponds to the Green's function on the whole real line; introducing BCs on $[0,1]$ amounts to introducing a perturbation on this free-space Green's function.
	\item A unified approach to obtain the Green's functions from the expansion of the resolvents. This is an expansion on old ideas due to Kato \cite{KATO}. 
\end{itemize}
The spectral questions explored here connect in a fundamental way with the works \cite{AGHKLT18,AGHKLT19,FGKLNS2021,FGKNS2021}. 

\begin{remark}
	We are interested in trace or sum rules for spectral functions associated with $\mathcal{K}$ in \cref{eq:Saito-Operator}, emanating from our recent work \cite{HS-ACHA18}. We plan to treat the closed forms for the iterated kernel and spectral functions corresponding to various Sturm-Liouville problems in our upcoming series of articles. \qed
\end{remark}

\begin{remark}
	Perturbation questions of the types treated here have been visited in the literature in applications to the Laplacian of a graph (in particular, a tree) and its connection with the distance matrix; see \cite{BapatKirklandNeumann,So}. 
	\qed
\end{remark}

The integral operator framework works very well for the nonlocal operator \cref{eq:Saito-Operator}. The literature has focused much on Dirichlet and Neumann BVPs; see, e.g., \cite{PorterStirling}. However, for the other BVPs we deal with in this paper, it is more difficult to derive the corresponding Green's functions.
We propose a unified approach to obtain the Green's functions from the expansion of the resolvents. This has been explored by Kato~\cite{KATO}, but not fully exploited yet as far as we know.

\begin{remark}
\Cref{tab:classif} summarizes the values of $\alpha$, $\beta$, and $\delta$ for each BC we deal with in our article, including the periodic and anti-periodic cases. They are coupled BCs and part of a different setting in the literature. \cite{KZ96b, KZ96a,KWZ2008}
\end{remark}

\begin{table}[tbhp]
	\footnotesize
	\caption{Discriminant and rank of perturbation for values of $\alpha, \beta, \delta$ for various BCs; KvN=Kre\u{\i}n-von-Neumann; GSARC=General Self-Adjoint Real-Coupled; GSARS=General Self-Adjoint Real-Separated;
		$\Delta:=\delta-2\beta-\alpha+\beta^2+\alpha \delta$}
	\hspace{-4em}
	\begin{tabular}{|l || c | c | c | c | c | c | c | c | c| c |c | c|}\hline
		\textbf{Problem}
		& 
		Nonlocal 
		& 
		KvN
		&
		Dirichlet
		& 
		Neumann
		& 
		Robin
		& 
		Radoux 
		&
		GSARC
		& 
		GSARC	&
		GSARS
  &
		GSARS
&
		Periodic
  &
		Anti-Periodic
  \\
		\hline\hline
		& 
		$\alpha=-1$
		& 
		$\alpha=-1$ 
		& 
		$\alpha \to  \infty$ 
		& 
		$\alpha=0$
		& 
		$\alpha \to \infty$
		& 
		$\alpha \to \infty$
		& 
		&
        &
        $\beta=0$
		&
        $\beta = 0$
		&
        &
        \\
		\textbf{Conditions}
		& 
		$\beta=-1$ 
		& 
		$\beta=1$
		& 
		$ \beta=0$
		& 
		$\beta=0$
		& 
		$\beta=0$
		& 
		$\beta=0$
		& 
		$\Delta\ne 0$
		& 
		$\Delta= 0$ 
	  & 
		$\Delta\ne 0$
		& 
		$\Delta= 0$ 
        &
        -
        &
        -
        \\		
		& 
		$\delta=1$ 
		& 
		$\delta=1$ 
		& 
		$\delta\to \infty$
		& 
		$\delta=0$ 
		& 
		$\delta=0$
		& 
		$\delta=1$ 
		&  
		&
		&
        &
        &
        &
		\\
		\hline
		
		\textbf{Value of $\Delta$}
		&
		$4$
		&
		$0$
		&
		$\infty$
		&
		$0$
		&
		$-\infty$
		&
		$1$
		&
		$\ne 0$ 
        &
		$0$
        &
		$\ne 0$
        &
		$0$
        &
		-
        &
        -
\\ \hline
		\textbf{Perturbation Rank}
		&
		-
		&
		$4$
		&
		$2$
		&
		$2$
		&
		$2$
		&
		$2$
		&
		$2$
		&
		$2,3, \text{or } 4$
		&
		$2$
		&
		$2$
	    &
        3
        &
        1
      \\
		\hline
	\end{tabular}
	\label{tab:classif}
\end{table}

The organization of this article is as follows. In \cref{sec:BVP-setting} we show the equivalence of the problems \cref{eq:Saito-Operator} and \cref{eq:Saito-Operator2}. 
In \cref{sec:resolvent-integral} we show how to obtain the integral operator from the resolvent. 
In \cref{sec:perturbation} we show that all the problems of the form \cref{eq:general} are lower rank perturbations of \cref{eq:Saito-Operator} of up to rank 4.

\section{From the Integral Operator to the Differential Equation}
\label{sec:BVP-setting}
To the best of our knowledge the properties of the operator \cref{eq:Saito-Operator} were first explored in a very general setting in the works of Marcel Riesz \cite{Riesz38, Riesz49}. For simplicity, let
\begin{equation}\label{eq:kappa}
\kappa(x, y) \define -\frac{1}{2} \, |x-y|.
\end{equation}
To see that \cref{eq:Saito-Operator} and \cref{eq:Saito-Operator2} are equivalent, observe, as in M.~Riesz (see formulas of \cite[p.~41]{Riesz38} and \cite[p.~73]{Riesz49}), that the function 
\begin{equation}\label{eq:fundamental}
u(x)= \int_{0}^{1} \kappa(x,y) f(y) \dd{y}
\end{equation}
is a solution of the Poisson equation 
\begin{equation}\label{eq:poisson}
-u''(x) = f(x).
\end{equation}
Note that 
\begin{equation} \label{eq:fund-soln}
-\frac{\partial^2 \kappa(x,y)}{\partial y^2}=\delta(x-y),
\end{equation}
where $\delta$ denotes the Dirac delta function. 
By \cref{eq:poisson} 
\begin{align} \label{eq:calc}
u(x) &= \int_{0}^{1} \kappa(x,y) \, f(y) \, \dd{y} \notag \\
&= -  \int_{0}^{1} \kappa(x,y) u''(y) \, \dd{y} \notag \\
&= - \int_{0}^{1} \frac{\partial^2 \kappa(x,y)}{\partial y^2}  u(y) \, \dd{y} + \left[-u'(y) \kappa(x,y)+ u(y) 
\frac{\partial \kappa}{\partial y} \right]_{0}^{1}  \notag \\
&= u(x)+ \left[-u'(y) \kappa(x,y)+ u(y) 
\frac{\partial \kappa}{\partial y} \right]_{0}^{1}.  
\end{align}
Therefore, it must satisfy the BC
\begin{equation*}
\left[-u'(y) \kappa(x,y)+ u(y) 
\frac{\partial \kappa}{\partial y} \right]_{0}^{1} = 0. 
\end{equation*}
This means 
\begin{equation*}
-u'(1) \kappa(x,1)+ u(1) 
\frac{\partial \kappa}{\partial y}\Big|_{y=1} = -u'(0) \kappa(x,0)+ u(0) 
\frac{\partial \kappa}{\partial y}\Big|_{y=0}
\end{equation*}
and reduces to
\begin{equation*}
-\frac{1}{2} u'(1)\, \left(x-1\right) + \frac{1}{2} u(1) = \frac{1}{2} u'(0)\,  x -\frac{1}{2} u(0) 
\end{equation*}
for any $x\in[0,1]$. Equating the coefficients leads the BCs in \cref{eq:Saito-Operator2}.

\section{From the Resolvent to the Integral Operator}
\label{sec:resolvent-integral}
The standard procedure to obtain the integral operator corresponding to a differential operator is to find the fundamental solution, or Green's function, first. Then the integral operator is just the formal solution of the problem. In our setting above, the fundamental solution of \cref{eq:Saito-Operator2} is $\kappa(x,y)$ (viz. \cref{eq:poisson}) satisfying its BCs. Reversing the steps of \cref{eq:calc} proves this point. 

Rather than following this procedure, we propose a unified framework for GSARC problems. They cover $SL(2,\mathbb{R})$ cases, and the prominent periodic and anti-periodic cases, which do not follow the GSARC setting. The various BCs are summarized in \cref{tab:classif}. 

We propose a framework where the calculations proceed via the \emph{resolvent}. The Green's function is then obtained from the expansion of this resolvent.
We note four well-known examples appearing in the literature: Stakgold and Holst \cite{StakgoldGreenBook} (Example 1, pp.~416--420, eigenvalue problem (7.1.28)); Kato \cite{KATO} (Example 6.21, p.~183; Example 4.14, p.~293; and Example 1.4, p.~367). In these examples, the resolvent is calculated explicitly, but the full extent of the method we are proposing in this article is not exploited.	

The proposed approach exploits in a fundamental way the Taylor and/or Laurent expansion of the resolvent function (or Neumann series) corresponding to linear operators. This approach aligns with the recent work of Gesztesy and Kirsten \cite{GK-survey19,GK-JFA19,GK-QAM19}. Our procedure explicitly calculates expressions of these operators for all GSARC problems. This method turns out to be an efficient tool for obtaining closed forms of the zeta spectral function in terms of the iterated Brownian bridge \cite{HSIteratedKernel2}. 

\subsection{Review of resolvents and Green's functions}
\label{sec:resolvent}
For relevant materials, we cite the standard literature on the properties of the resolvent, and provide the necessary background we will need in this article; see Yosida \cite[Chap.~8]{Yosida}; Kato \cite[Chap.~1; Chap.~10]{KATO}; Dunford \& Schwartz \cite[Chap.~7]{Dunford-Schwartz}; and Taylor \& Lay \cite[Chap.~5]{Taylor-Lay}. These expansions have roots in Nagumo's 1916 article \cite{Nagumo}; see also \cite{KRESS3}.  
Let $A$ be a densely defined and closed linear operator of trace class on a Hilbert space $\mathcal{H}$ with domain $\mathcal{D}(A)$ and range $\mathcal{R}(A)$ in $\mathcal{H}$. For the applications we envision in this paper, the \emph{spectrum} consists of a discrete set of eigenvalues, $\lambda_1\le \lambda_2\le \ldots \le \lambda_k\le \ldots $ accumulating to $\infty$, i.e., $\sigma(A)=\{\lambda_k\}_{k=1}^{\infty}$ This is certainly the case of all the problems defined by the eigenvalue problem \cref{eq:general}. Finally, let  $\rho(A)=\mathbb{C}\setminus \sigma(A)$ be the \emph{resolvent set} of $A$, and $T=A^{-1}$ its inverse. 

On $\rho(A)$, the \emph{resolvent} $R(z) \define \left(A-z I\right)^{-1}$ is a 
holomorphic 
function of $z$.
Moreover, it admits the expansion (see, e.g., \cite[p.~23; p.~193]{KRESS3}):

\begin{align} \label{eq:Neumann-series}
R(z) &= \left(A \, \left(I-z \, A^{-1}\right)\right)^{-1} \notag \\
&=\left(I-z \, A^{-1}\right)^{-1} \, A^{-1} \notag \\
&= A^{-1} + A^{-2} z + A^{-3} \, z^2 + \cdots \notag \\
&= T + T^{2} z + T^{3} \, z^2 + \cdots
\end{align}
for $|z| \, \|A^{-1}\|= |z| \, \|T\| <1$. Here $I$ is the identity operator	on $\mathcal{H}$. 
In the applications we propose $T$  will correspond to the integral operator associated with $A$, while $T^{k+1}$ is  the associated iterated form of order $k$. 
The singularities of $R(z)$ are exactly the eigenvalues of $A$. If $z=0$ is an eigenvalue, $R(z)$ admits \emph{the Laurent series}  
\begin{equation} \label{eq:Laurent-series}
R(z) = \sum_{n=-\infty}^{\infty} z^n A_n. 
\end{equation}
The coefficients $A_n$ are given by
\begin{equation} \label{eq:Laurent-series2}
A_n = \frac{1}{2 \pi \im} \, \oint_{C(0, \epsilon)} z^{-n-1} R(z) \dd{z}.  
\end{equation}
where $C(0,\epsilon)$ is a positively-oriented small circle of radius $\epsilon$ centered at $z=0$, excluding all other eigenvalues of $A$. We note that \emph{the Riesz projection} $P:=-A_{-1}$. In more explicit terms, following Kato \cite[Sec.~III.6.5]{KATO}, if $z=0$ is an isolated singularity, then
\begin{equation} \label{eq:Laurent-series-key}
R(z) = -\frac{P}{z} - \sum_{n=1}^{\infty} \frac{D^n}{z^{n+1}}+ 
\sum_{n=0}^{\infty} z^n S^{n+1},
\end{equation}
with 
\begin{equation} \label{eq:nilpotent}
P \define-\frac{1}{2\pi \im} \, \oint_{C(0, \epsilon)} R(z) \dd{z}, \quad D \define A P = -\frac{1}{2\pi \im} \, \oint_{C(0, \epsilon)} z R(z) \dd{z} ,
\end{equation}
$D$ is quasi-nilpotent with 
\begin{equation}
D=D\, P = P \, D ,
\end{equation}
and
\begin{equation}
S := \frac{1}{2\pi \im} \, \oint_{C(0, \epsilon)} \frac{R(z)}{z} \dd{z} ,
\end{equation}
satisfying 
\begin{equation}
A \, S = I - P,  \quad S \, P = P \, S = 0. 
\end{equation}
The dimension of the range corresponding to $P$ is the algebraic multiplicity, $m(T;0)$ of $z=0$ \cite{GK-survey19,GK-JFA19,GK-QAM19}:
\begin{equation}
m(T;0)= \dim \mathcal{R}(P) = \tr_{\mathcal{H}} \, \left(P \right) .
\end{equation}
When $z=0$ is not an isolated singularity, $P\equiv 0$, $S=A^{-1}=T$ and we recover \cref{eq:Neumann-series}.


\subsection{Our one-dimensional setting}
\label{sec:1D}
For the BVPs we deal with, we set $A=-\difh{2}{}{x}$ + BCs on $[0,1]$. We present a unified approach to all GSARC problems. We also investigate a couple of non-GSARC problems that appear in the literature, namely the periodic and anti-periodic cases treated in \cite{FGKLNS2021}. Our results in this article focus on perturbation questions \emph{not} pursued in \cite{FGKLNS2021,GK-survey19,GK-JFA19,GK-QAM19}. 

For $z\in \rho(A)$, the resolvent $R(z)$ applied to a function $f(x)$ takes the form
\begin{equation}\label{eq:Green-Helmholtz}
R(z)f(x) = \left(\left(A- z I\right)^{-1}\, f \right) (x)=\int_{0}^{1} G(z,x,y) f(y) \, \dd{y} ,
\end{equation}
where $G(z,x,y)$ is the associated kernel.
Explicit resolvent kernels for the various BCs considered in this paper are listed in \cref{tab:Resolvent}. We will also denote by $G_0(x,y)$ 
the Green's function corresponding to the expansion \cref{eq:Laurent-series-key} 
\begin{equation}\label{eq:Green-S}
S f (x) = \left(A^{-1}\left(I- P\right) f \right) (x)=\int_{0}^{1} G_0(x,y) f(y) \, \dd{y}.
\end{equation}
It is explicitly given in \cref{tab:Green} for the various BCs considered in this paper. 
\begin{remark}
	Though its study is relegated to our subsequent article, the iterated Green's function corresponding to $S^{n+1}$ in \cref{eq:Laurent-series-key} is denoted by $G_n(x,y)$, viz., 
	\begin{equation}\label{eq:Green-Sn}
	S^{n+1} f (x) = \int_{0}^{1} G_n(x,y) f(y) \, \dd{y}.
	\end{equation} \qed
\end{remark}
We end this discussion by deriving the condition under which \cref{eq:general} admits $\lambda=0$ as an eigenvalue (and hence leads to an isolated singularity leading to the Laurent series \cref{eq:Laurent-series-key} for the resolvent.) Clearly, for $\lambda=0$, we are after the condition that all lines of the form $y=m x +b$ are non-trivial solutions of the BVP \cref{eq:general} (with $\gamma=-\beta$). This means
\begin{align*}
m &=  \alpha b+ \beta (m+b)\notag \\
m &= -\beta b+ \delta (m+b)
\end{align*}
or 
\begin{align*} 
(1-\beta) m - (\alpha + \beta) b &= 0\notag \\
(1-\delta) m +(\beta- \delta) b &= 0. 
\end{align*}
Hence the following secular equation must hold:
\begin{equation}\label{eq:discriminant-condition}
\begin{vmatrix}1-\beta & - (\alpha + \beta) \\  1-\delta & \beta- \delta \end{vmatrix}=0.
\end{equation}
Simplifying, we are led to the condition $\Delta:= \delta - \alpha - 2 \beta + \beta^2 + \alpha \, \delta = 0$. We call $\Delta$ the \emph{discriminant} of the corresponding Sturm-Liouville problem.

Note that when $\beta=1$, \cref{eq:discriminant-condition} leads to $\alpha=-1$ \emph{or} $\delta=1$. This is a generalized instance of the Kre\u{\i}n-von Neumann problem where $\alpha=-1$, $\beta=1$, \emph{and} $\delta=1$. In fact, for the BVPs we deal with, only the Kre\u{\i}n-von Neumann BVP belongs to the GSARC class with $\Delta = 0$ while the other BVPs belong to the GSARC class with $\Delta \neq 0$. 
The discriminant condition and various BCs, including the ``separated'' case ($\beta=0$) are plotted in Fig.~\ref{fig:Discr Cond}. 

\begin{center}
\begin{figure}[tbhp!]
    \centering
    \includegraphics[width= .8\textwidth]{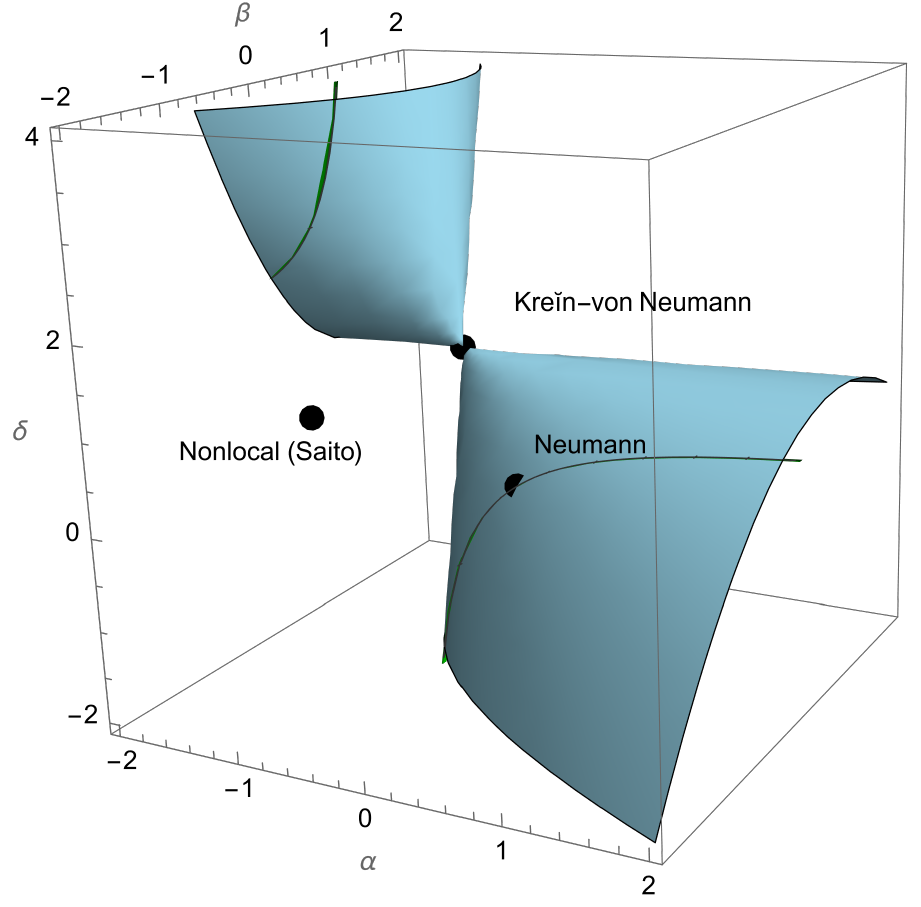}
    \caption{Discriminant condition $\delta - \alpha - 2 \beta + \beta^2 + \alpha \, \delta = 0$ corresponding to $\lambda=0$ being an eigenvalue, with various BCs treated, including the discriminant condition for the ``separated'' BCs case ($\beta=0$) indicated by the curve on the two sheets of the surface.}
    \label{fig:Discr Cond}
\end{figure}
\end{center}

\begin{table}[tbhp]
	\caption{Resolvent kernel $G(z,x,y)$ for various BCs; the formulas for GSARC  with  $\Delta=0$, GSARS with $\Delta=0$, and GSARS with $\Delta \neq 0$ are too long to list in this table; see \cref{Appendix:A}, \cref{Appendix:C}, and \cref{sec:GSARS Boundary Conditions}, respectively}
	\footnotesize
	\begin{tabular}{|c|c|} \hline
		\textbf{Problem}
		& \textbf{\begin{tabular}[c]{@{}c@{}} Resolvent \end{tabular}} \\ \hline \hline
		
		Nonlocal (Saito)
		& $ \frac{1}{ \sqrt{z}\, \(2+ 2 \cos\sqrt{z}+ \sqrt{z} \, \sin\sqrt{z} \)}
		\begin{cases}
		- \sqrt{z} \, \cos\(\sqrt{z}x\) \, \cos\(\sqrt{z}(1-y)\)
		+ 2 \cos\frac{\sqrt{z}}{2} \, \sin\(\frac{\sqrt{z}}{2}(1+2x-2y)\), & 0 \leq x \leq y \leq 1 \\
		- \sqrt{z} \, \cos\(\sqrt{z}y\) \, \cos\(\sqrt{z}(1-x)\)
		+ 2 \cos\frac{\sqrt{z}}{2} \, \sin\(\frac{\sqrt{z}}{2}(1+2y-2x)\) & 0 \leq y \leq x \leq 1
		\end{cases} $ \\ \hline
		
		Kre\u{\i}n-von Neumann
		\begin{hide}
			& $ \frac{1}{4 \sqrt{z} \( \sqrt{z} \cos\frac{\sqrt{z}}{2} - 2 \sin\frac{\sqrt{z}}{2} \) \sin\frac{\sqrt{z}}{2}}
			\begin{cases}
			\(-2 \sqrt{z} \cos\(\sqrt{z}x\) \, \cos\(\sqrt{z}(1-y)\) - 4 \sin\frac{\sqrt{z}}{2} \sin\(\sqrt{z}x\) \sin\(\frac{\sqrt{z}}{2}\(1-2y\)\) \\
			\hspace{1em} + 4 \sin\frac{\sqrt{z}}{2} \cos\(\sqrt{z}x\) \cos\(\frac{\sqrt{z}}{2} \(1-2y\)\) \), & 0 \leq x \leq y \leq 1 \\
			
			\(-2 \sqrt{z} \cos\(\sqrt{z}y\) \, \cos\(\sqrt{z}(1-x)\) - 4 \sin\frac{\sqrt{z}}{2} \sin\(\sqrt{z}y\) \sin\(\frac{\sqrt{z}}{2}\(1-2x\)\) \\
			\hspace{1em} + 4 \sin\frac{\sqrt{z}}{2} \cos\(\sqrt{z}y\) \cos\(\frac{\sqrt{z}}{2}\(1-2x\)\) \), & 0 \leq y \leq x \leq 1   
			\end{cases} $ \\ \hline
		\end{hide}
		& $ \frac{1}{ \sqrt{z}\, \(-2 + 2 \cos\sqrt{z}+ \sqrt{z} \, \sin\sqrt{z} \)}
		\begin{cases}
		- \sqrt{z} \, \cos\(\sqrt{z}x\) \, \cos\(\sqrt{z}(1-y)\)
		+ 2 \sin\frac{\sqrt{z}}{2} \, \cos\(\frac{\sqrt{z}}{2}(1+2x-2y)\), & 0 \leq x \leq y \leq 1 \\
		- \sqrt{z} \, \cos\(\sqrt{z}y\) \, \cos\(\sqrt{z}(1-x)\)
		+ 2 \sin\frac{\sqrt{z}}{2} \, \cos\(\frac{\sqrt{z}}{2}(1+2y-2x)\) & 0 \leq y \leq x \leq 1
		\end{cases} $ \\ \hline

		Dirichlet~\cite{AGHKLT18}
		& $ \frac{1}{\sqrt{z} \, \sin\sqrt{z}}
		\begin{cases}
		\sin\(\sqrt{z}x\) \, \sin\(\sqrt{z}(1-y)\), & 0 \leq x \leq y \leq 1 \\
		\sin\(\sqrt{z}y\) \, \sin\(\sqrt{z}(1-x)\), & 0 \leq y \leq x \leq 1   
		\end{cases} $ \\ \hline
		
		Neumann
		& $ -\frac{1}{\sqrt{z} \sin\sqrt{z}}
		\begin{cases}
		\cos\(\sqrt{z}x\) \, \cos\(\sqrt{z}(1-y)\), & 0 \leq x \leq y \leq 1 \\
		\cos\(\sqrt{z}y\) \, \cos\(\sqrt{z}(1-x)\), & 0 \leq y \leq x \leq 1   
		\end{cases} $ \\ \hline    
		
		Robin
		& 
		$ \frac{1}{\(\alpha - \delta \) \, z \cos\sqrt{z} - \sqrt{z} \(z+ \alpha \delta \) \, \sin\sqrt{z} }
		\begin{cases}
		\(\sqrt{z}\, \cos \(\sqrt{z} x\) + \alpha \, \sin \(\sqrt{z} x \) \) \, \(\sqrt{z}\, \cos \(\sqrt{z} (1-y)\) - \delta \ \sin \(\sqrt{z} (1-y)\) \) & 0 \leq x \leq y \leq 1 \\
		\(\sqrt{z}\, \cos \(\sqrt{z} y\) + \alpha \, \sin \(\sqrt{z} y \) \) \, \(\sqrt{z}\, \cos \(\sqrt{z} (1-x)\) - \delta \ \sin \(\sqrt{z} (1-x)\) \) & 0 \leq y \leq x \leq 1
		\end{cases} $ \\ \hline
		
		Periodic
		& $ -\frac{1}{2\sqrt{z} \sin\frac{\sqrt{z}}{2}}
		\begin{cases}
		\cos\(\frac{\sqrt{z}}{2}(1+2x-2y)\), & 0 \leq x \leq y \leq 1 \\
		\cos\(\frac{\sqrt{z}}{2}(1+2y-2x)\), & 0 \leq y \leq x \leq 1   
		\end{cases} $ \\ \hline
		
		Anti-Periodic
		& $ \frac{1}{2 \sqrt{z} \cos\frac{\sqrt{z}}{2}}
		\begin{cases}
		\sin\(\frac{\sqrt{z}}{2}(1+2x-2y)\), & 0 \leq x \leq y \leq 1 \\
		\sin\(\frac{\sqrt{z}}{2}(1+2y-2x)\), & 0 \leq y \leq x \leq 1   
		\end{cases} $ \\ \hline
		
		Radoux                                                                             & $ -\frac{1}{\sqrt{z} \(\sqrt{z} \cos\sqrt{z} - \sin\sqrt{z} \)}
		\begin{cases}
		\sqrt{z} \sin\(\sqrt{z}x\) \, \cos\(\sqrt{z}(1-y)\) - \sin\(\sqrt{z}(1-y)\), & 0 \leq x \leq y \leq 1 \\
		\sqrt{z} \sin\(\sqrt{z}y\) \, \cos\(\sqrt{z}(1-x)\) - \sin\(\sqrt{z}(1-x)\), & 0 \leq y \leq x \leq 1   
		\end{cases} $ \\ \hline
		
		\begin{tabular}[c]{@{}c@{}}GSARC, $\Delta\ne 0$\end{tabular}
		& $ \frac{1}{\(-4 \beta z + 2z \(-\alpha+ \delta\)\cos\sqrt{z} + 2 \sqrt{z} \(z+ \beta^2+ \alpha \delta\)\sin\sqrt{z}\)}
		\begin{cases}
		-\(z+ \beta^2 + \alpha \delta\) \cos\(\sqrt{z}(1+x-y)\) + \(-z+ \beta^2 + \alpha \delta\) \cos\(\sqrt{z}(1-x-y)\) \\
		\hspace{1em} + \sqrt{z} \big(
		- 2\beta \sin\(\sqrt{z}(x-y)\)
		+ \(-\alpha + \delta\) \sin\(\sqrt{z}(1+x-y)\) \\
		\hspace{2em} + \(\alpha + \delta\) \sin\(\sqrt{z}(1-x-y)\) \big),
		& \hspace{-5em} 0 \leq x \leq y \leq 1 \\
		-\(z+ \beta^2 + \alpha \delta\) \cos\(\sqrt{z}(1+y-x)\) + \(-z+ \beta^2 + \alpha \delta\) \cos\(\sqrt{z}(1-y-x)\) \\
		\hspace{1em} + \sqrt{z} \big(
		- 2\beta \sin\(\sqrt{z}(y-x)\)
		+ \(-\alpha+ \delta\) \sin\(\sqrt{z}(1+y-x)\) \\
		\hspace{2em} + \(\alpha+ \delta\) \sin\(\sqrt{z}(1-y-x)\) \big),
		& \hspace{-5em} 0 \leq y \leq x \leq 1   
		\end{cases} $ \\ \hline
	\end{tabular}
	\label{tab:Resolvent}
\end{table}

\begin{table}[tbhp]
	\footnotesize
	\caption{Green's function $G(x,y)$ for various BCs; see \cref{Appendix:A} and \cref{Appendix:C} for the formulas of GSARC and GSARS with $\Delta=0$}
	\begin{tabular}{|c|c|}\hline
		\textbf{Problem}
		& \textbf{\begin{tabular}[c]{@{}c@{}} Green's function \end{tabular}} \\ \hline \hline
		Nonlocal (Saito)
		& $ -\frac{1}{2} |x-y| $ \\ \hline
		
		Kre\u{\i}n-von Neumann
		& $-\frac{1}{2} |x-y| +\frac{2}{15}- \frac{3}{5} \(x+y\) + 2   \(x^2+y^2\) +\frac{6}{5} x y - 3 x y \(x+y\) - \(x^3+y^3 \) +2 x y \(x^2+y^2\)$ \\ \hline
		
		Dirichlet~\cite{AGHKLT18}
		& $ -\frac{1}{2} |x-y| +\frac{1}{2} \left(x+y\right) - x y $ \\ \hline
		
		Neumann
		& $-\frac{1}{2} |x-y| +\frac{1}{3} - \frac{1}{2} \left(x+y\right)+ \frac{1}{2} \left(x^2+y^2\right) $ \\ \hline
		
		Robin
		& $ -\frac{1}{2} |x-y| -\frac{1}{2 \(\delta - \alpha + \alpha \delta\)} \, \(2 - 2 \delta + \(\alpha+\delta - \alpha \delta\)  \(x+y\)  + 2 \alpha \delta x y \) $ \\ \hline
		
		Periodic
		& $ -\frac{1}{2} |x-y| +\frac{1}{12}+\frac{1}{2} \(x-y\)^2$\\ \hline
		
		Anti-Periodic
		&  $ -\frac{1}{2} |x-y| +\frac{1}{4}$ \\ \hline
		
		Radoux
		& $ -\frac{1}{2} |x-y| +\frac{1}{2} \(x+y\) - \frac{9}{5}x y + \frac{1}{2} x y \(x^2+y^2\) $ \\ \hline
		
		\begin{tabular}[c]{@{}c@{}} GSARC, $\Delta\neq 0$ \end{tabular}
		& $  -\frac{1}{2} |x-y| +\frac{1}{ \beta^2 + \alpha \delta-\alpha+ 2 \beta+ \delta} \left( \(-1+\delta\) +\frac{1}{2} \(\beta^2+ \alpha \delta  -\alpha - \delta \) \(x+y\) - \(\beta^2+ \alpha \delta \) x y \right) $ \\ \hline
\begin{tabular}[c]{@{}c@{}}GSARS, $\Delta\ne 0$\end{tabular}
		&  $ -\frac{1}{2} |x-y|-\frac{\left(\cos\theta_0 \cos \theta_1 + \sin (\theta_0+\theta_1)\right) (x+y) - 2 x y \cos \theta_0 \cos \theta_1 }{\cos (\theta_0 - \theta_1) + \cos (\theta_0+\theta_1) - 2 \sin (\theta_0 - \theta_1)}$ \\ \hline	\end{tabular}
	\label{tab:Green}
\end{table}

\subsection{The Nonlocal Boundary Conditions \texorpdfstring{\cref{eq:Saito-Operator2}}{}} 

This is the eigenvalue problem \cref{eq:Saito-Operator2}, which corresponds to the BVP \cref{eq:general} with 
\begin{equation*}
C = \left( {\begin{array}{cc}
	-1 & -1 \\
	1 & 1 \\
	\end{array} }\right)
\end{equation*}
with discriminant $\Delta=4.$ Since $\lambda = 0$ is not an eigenvalue, its Riesz projection $P \equiv 0$. To obtain the expression of the resolvent kernel in \cref{tab:Resolvent}, we first fix $0\le y\le 1$, and let
\begin{equation} \label{eq:form-soln}
u(x,y) := 	\begin{cases}
u_1(x,y) & 0 \leq x \leq y \leq 1 \\
u_2(x,y) & 0 \leq y \leq x \leq 1
\end{cases}
\end{equation}
be a solution of the Helmholtz equation 
\begin{equation} \label{eq:Helm}
u''+ z \, u=-\delta(x-y)
\end{equation}
with $'$ denote the derivative $\dif{}{x}$. 
Here, for simplicity, $G(z, x, y)=u(x,y)$, suppressing the dependence on $z$ in the expressions of $u_1$ and $u_2$. 
We now impose the BCs
\begin{equation}\label{eq:S-BC}
\begin{split}
u_1'(0,y) &= -u_1(0,y)- u_2(1,y) \\
u_2'(1,y) &= u_1(0,y)+ u_2(1,y)  \\
u_1(y^{-},y) &= u_2(y^{+},y) \\
u_2'(y^{+},y)-u_1'(y^{-},y) &= -1
\end{split}
\end{equation}
The last two equations of \cref{eq:S-BC} reflect, respectively, the continuity and jump condition at $y$ (obtained by integrating \cref{eq:Helm}). 
It is a common strategy~\cite{STRAUSS} to set
\begin{equation} \label{eq:form-soln2}
\begin{split}
u_1(x,y) &= a_1 \cos \(\sqrt{z} x\) + a_2 \sin \(\sqrt{z} x\), \quad 0 \leq x \leq y \leq 1 ; \\  
u_2(x,y) &= b_1 \cos \(\sqrt{z} x\) + b_2 \sin \(\sqrt{z} x\),  \quad 0 \leq y \leq x \leq 1 ,
\end{split}
\end{equation}
and then plug \cref{eq:form-soln2} into \cref{eq:S-BC} to determine $a_1, a_2, b_1$, and $b_2$. These then lead to a solution $G(z,x,y)$, which depends on the spectral parameter $z$, as well as $x$ and $y$: 
\begin{equation} \label{eq:Gz-saito}
G(z,x,y)= \frac{- \sqrt{z} \, \cos\(\sqrt{z}x\) \, \cos\(\sqrt{z}(1-y)\)
	+ 2 \cos\frac{\sqrt{z}}{2} \, \sin\(\frac{\sqrt{z}}{2}(1+2x-2y)\)}{ \sqrt{z}\, \(2+ 2 \cos\sqrt{z}+ \sqrt{z} \, \sin\sqrt{z} \)},
\end{equation}
for $0 \leq x \leq y \leq 1$, and by symmetry, the expression $G(z,y,x)$ when $0 \leq y \leq x \leq 1$. The poles of $G(z,x,y)$ as a function of $z$ are exactly the eigenvalues of our differential operator. Note that $G_0(x,y) = \lim_{z\to 0^{+}} \, G(z,x,y)=\frac{1}{2} (x-y)$, for $0 \leq x \leq y \leq 1$, thus recovering the expression of the integral operator \cref{eq:Saito-Operator}. Here, we use the convention of Kato \cite{KATO}, for $0\le y\le x\le 1$, i.e., $G_0(x,y) = \frac{1}{2}(y-x)$ by symmetry. The full expressions for all BVPs considered in this paper are listed in \cref{tab:Resolvent,tab:Green}. 
Explicit expressions for the kernel of the Riesz projection, $p(x,y)$, are displayed in \cref{tab:RieszProj} for various conditions.

\begin{table}[tbhp]
	\footnotesize
	\caption{Riesz Projection kernel for various BCs; the case of GSARC, $\Delta=0$, $\alpha=-1$ reduces to a generalized Kre\u{\i}n-von Neumann problem (with $\delta$ as a free parameter); see \cref{rk:III.8} for further details}
	\begin{center}
		\begin{tabular}{|c|c|}\hline
			\textbf{Problem}
			& \textbf{\begin{tabular}[c]{@{}c@{}} Riesz Projection Kernel \end{tabular}} \\ \hline \hline
			Nonlocal (Saito)
			& $0$ (None) \\ \hline
			
			Kre\u{\i}n-von Neumann
			& $ -4 + 6 x+ 6 y - 12 x y  $ \\ \hline
			
			Dirichlet~\cite{AGHKLT18}
			& $0$ (None)  \\ \hline
			
			Neumann
			& $-1 $ \\ \hline
			
			Robin
			& $0$ (None) \\ \hline
			
			Periodic
			& $-1$ \\ \hline
			
			Anti-Periodic
			&  $0$ (None) \\ \hline
			
			Radoux
			& $ - 3 x y $ \\ \hline
			
			\begin{tabular}[c]{@{}c@{}} GSARC, $\Delta \ne 0$\end{tabular}
			& $0$ (None) \\ \hline
			
			\begin{tabular}[c]{@{}c@{}} GSARC, $\Delta=0$ and $\alpha \neq -1$ \end{tabular}
			& $ \frac{- 3 \,  \left(1-\beta + (\alpha +\beta) x\right) \, \left(1-\beta + (\alpha +\beta) y\right)}{3+ \alpha^2+\beta^2+ 3 \alpha- 3 \beta - \alpha \beta}$  \\ \hline	
   \begin{tabular}[c]{@{}c@{}}GSARS, $\Delta\ne 0$\end{tabular}
		&   $0$ (None) \\ \hline
   \begin{tabular}[c]{@{}c@{}}GSARS, $\Delta= 0$\end{tabular}
		&   $\frac{6 \left(x \cos \theta_0 - \sin \theta_0\right) \, \left(y \cos \theta_0 - \sin \theta_0\right)}{-4 + 2 \cos 2 \theta_0+ 3 \sin 2 \theta_0}$ \\ \hline
		\end{tabular}
	\end{center}
	\label{tab:RieszProj}
\end{table}

\subsection{Kre\u{\i}n-von Neumann Boundary Conditions}
\label{sec:KvN Resolvent}
This problem received special attention in the literature, both in one dimension, and in higher dimensions, \cite{AlonsoSimon80,AGHKLT19,AGHKLT18,AGLMS2017,AGMST2010,AGMST2013}, partially in connection with work on self-adjoint extensions of the operator $\mathcal{T_{\infty}}=-\difh{2}{}{x}$. Note that the Kre\u{\i}n-von Neumann extension is the smallest (or soft) positive self-adjoint extension while the Friedrichs extension is the largest (or hard) positive self-adjoint extension. For the definitions and relevant characterizations we refer the reader to the cited literature as well as \cite{SCHMUEDGEN}.

In one dimension, it corresponds to the problem \cref{eq:general} with 
\begin{equation*}
C = \left( {\begin{array}{cc}
	-1 & 1 \\
	-1 &1 \\
	\end{array} }\right)
\end{equation*}
with discriminant $\Delta=0.$ Strauss \cite[pg.~101, Exercise 12; pg.~145, Exercise 4]{STRAUSS} calls this problem an ``unusual'' eigenvalue problem. We owe this remark to M.~Ashbaugh who brought it to our attention. Solving \cref{eq:Helm} for $u$ given by \cref{eq:form-soln} (and \cref{eq:form-soln2}), and imposing the boundary, continuity, and jump conditions
\begin{equation}\label{eq:S-BC2}
\begin{split}
u_1'(0,y) &=-u_1(0,y)+ u_2(1,y) \\
u_2'(1,y) &=-u_1(0,y)+ u_2(1,y) \\
u_1(y^{-},y) &= u_2(y^{+},y) \\
u_2'(y^{+},y)-u_1'(y^{-},y) &=-1
\end{split}
\end{equation}
gives
\begin{hide}
	\small{
		\begin{align*} 
		G(z,x,y)=\frac{-2 \sqrt{z} \cos\(\sqrt{z}x\) \, \cos\(\sqrt{z}(1-y)\) - 4 \sin\frac{\sqrt{z}}{2} \sin\(\sqrt{z}x\) \sin\(\frac{\sqrt{z}}{2}\(1-2y\)\) + 4 \sin\frac{\sqrt{z}}{2} \cos\(\sqrt{z}x\) \cos\(\frac{\sqrt{z}}{2} \(1-2y\)\)}{4 \sqrt{z} \( \sqrt{z} \cos\frac{\sqrt{z}}{2} - 2 \sin\frac{\sqrt{z}}{2} \) \sin\frac{\sqrt{z}}{2}}
		\end{align*}
	}
\end{hide}
\begin{align*} 
G(z,x,y)=\frac{- \sqrt{z} \cos\(\sqrt{z}x\) \, \cos\(\sqrt{z}(1-y)\) + 2 \sin\frac{\sqrt{z}}{2} \cos\(\frac{\sqrt{z}}{2}(1+2x-2y)\)}{ \sqrt{z} \( -2 + 2\cos\sqrt{z} + \sqrt{z} \sin\sqrt{z}\)}
\end{align*}
for $0 \leq x \leq y \leq 1$, and by symmetry, the expression $G(z,y,x)$ when $0 \leq y \leq x \leq 1$. The kernel of the Riesz projection is given by 
\begin{equation}\label{eq:RiezKvN}
p(x,y) := \lim_{z\to 0^+} z \, G(z, x,y) = -4 + 6 x+ 6 y - 12 x y.
\end{equation} 
The quasi-nilpotent operator $D$ defined in \cref{eq:nilpotent} is null in this case because $$\lim_{z\to 0^{+}} z^2 \, G(z, x,y)=0.$$ 
The Green's function corresponding to this problem is obtained by finding the limit 
\begin{align*} \label{eq:integral-KvN2}
G_0(x,y) &=\lim_{z\to 0^{+}} \, G(z,x,y) - \frac{p(x,y)}{z} \\ 
&= \frac{1}{30} \big(
4 - 3 x - 33 y+60 x^2+ 60 y^2 + 36 xy 
-30 x^3-90 x^2 y - 90 x y^2- 30 y^3 \\[0.5em]
&  \hspace{1em} + \, 60 x^3 y + 60 x y^3\big)
\end{align*}
for $0\le x\le y\le 1$. 
A unified expression, valid for all $x, y \in[0,1]$, and relating to the nonlocal operator appears in \cref{tab:Green}.

\subsection{Dirichlet Boundary Conditions}
\label{sec:Dirichlet Resolvent}
This example is found in many sources, e.g., \cite{Dikii61,GohbergKrein1969,KATO,PorterStirling,SAITO-LAPEIG-ACHA,CavFassMcC15,AGHKLT18}.
It is the easiest as it captures the essence of many features of GSARC problems. 
As a GSARC problem, this is the limiting case when $\alpha \to \infty$, $\beta=-\gamma =0$, and $\delta\to \infty$ (e.g., $\alpha=\delta\to \infty$, thus $\Delta\to \infty$. ).  

Solving \cref{eq:Helm} for $u$ given by \cref{eq:form-soln} (and \cref{eq:form-soln2}), and imposing the boundary, continuity, and jump conditions
\begin{equation}\label{eq:S-BC3}
\begin{split}
u_1(0,y) &= 0 \\
u_2(1,y) &= 0 \\
u_1(y^{-},y) &= u_2(y^{+},y) \\
u_2'(y^{+},y)-u_1'(y^{-},y) &= -1
\end{split}
\end{equation}
yields
\begin{align*} 
G(z,x,y)=\frac{\sin\(\sqrt{z}x\) \, \sin\(\sqrt{z}(1-y)\)}{\sqrt{z} \, \sin\sqrt{z}} .
\end{align*}
Note that $\lim_{z\to 0^{+}} z\, G(z,x,y) = 0$. Thus the Riesz projection $P\equiv 0$ and $G(z,x,y)$ admits a pure Taylor series at $z=0$, from which we get 
\begin{align*}
G_0(x,y)=\lim_{z\to 0^{+}}  G(z,x,y) =x (1-y),
\end{align*}
for $0\le x\le y\le 1$, and the symmetric expression for  $0\le y\le x\le 1$. This fundamental solution assumes the unified form
\begin{equation*}
    G_0(x,y)=-\frac{1}{2} |x-y| +\frac{1}{2} \left(x+y\right) - x y,
\end{equation*}
valid for all $x, y \in[0,1]$. The resolvent and Green's function are tabulated in \cref{tab:Resolvent,tab:Green} (see also \cite{Hellwig,KATO}).

\subsection{Neumann Boundary Conditions}
\label{sec:Neumann Resolvent}
The expression for the Green's function $G_0(x,y)$ corresponding to this eigenvalue problem appear, for example, in \cite{PorterStirling,SAITO-LAPEIG-ACHA}. To obtain the resolvent kernel and Green's function expressions in \cref{tab:Resolvent,tab:Green}, we proceed as before, with the associated GSARC matrix $C\equiv 0$, and the discriminant $\Delta=0.$ For $0\le x\le y\le 1$, the resolvent kernel is explicitly given by 
\begin{align*} 
G(z,x,y)=-\frac{\cos\(\sqrt{z} x\) \, \cos\(\sqrt{z}(1-y)\) }{\sqrt{z} \sin\sqrt{z}} .
\end{align*}
The kernel of Riesz projection $P$ is given by $p(x,y)=\lim_{z\to 0^{+}} z \, G(z, x,y)=-1$, while the nilpotent operator $D\equiv 0$ (since $\lim_{z\to 0^{+}} z^2 \, G(z, x,y)=0$). The singularity $z=0$ corresponds to a simple eigenvalue of the problem \cref{eq:general}. The Green's function corresponding to this problem is then  
\begin{align*}
G_0(x,y) &= \lim_{z\to 0^{+}} \, G(z,x,y) - \frac{p(x,y)}{z} \\ 
&= \frac{1}{6} 
\(2 - 6 y+ 3 x^2 + 3 y^2\),
\end{align*}
for $0\le x\le y\le 1$. This agrees with the expressions in \cite{PorterStirling,SAITO-LAPEIG-ACHA}.

\subsection{Robin Boundary Conditions}
\label{sec:Robin Resolvent}
In the Robin boundary problem (named after the French mathematician Victor Gustave Robin; see, e.g., \cite{ROBIN1, ROBIN2}), we still deal with the eigenvalue problem \cref{eq:general} but now subject to (the separated) conditions where the flux at the boundary is proportional to its value, i.e., 
$u'(0) = \alpha u(0)$ and $u'(1) =\delta u(1)$. 
In the recent literature the problem is sometimes called the ``Generalized Robin Problem''  \cite{GMGeneralizedRobin2008}.
We will first focus on the general case, then treat particular cases which received special attention in the literature \cite{KATO,StakgoldGreenBook}. The discriminant is given by $\Delta=\delta-\alpha  + \alpha \, \delta$. We will treat the case when $\Delta \neq 0$, relegating the null case to the most general setting.  We solve for the resolvent kernel as before, to get, for $0\le x \le y \le 1$
\begin{align*} 
G(z,x,y)=\frac{\(\sqrt{z}\, \cos \(\sqrt{z} x\) + \alpha \, \sin \(\sqrt{z} x \) \) \, \(\sqrt{z}\, \cos \(\sqrt{z} (1-y)\) - \delta \ \sin \(\sqrt{z} (1-y)\) \)}{\(\alpha - \delta \) \, z \cos\sqrt{z} - \sqrt{z} \(z+ \alpha \delta \) \, \sin\sqrt{z} },
\end{align*}
and the symmetric expression $G(z, y,x)$ for $0\le y\le x \le 1$. \\
Since $\lim_{z\to 0^{+}} z \, G(z, x,y)=0$, the Riesz projection $P$ and nilpotent operator $D$ are identically equal to zero. 
For $0\le x \le y \le 1$, the Green's function corresponding to this problem is given by  
\begin{align*}
G_0(x,y) &= \lim_{z\to 0^{+}} \, G(z,x,y) \notag\\ 
&= - \frac{\( 1+ \alpha x\) \(1- \delta (1-y)\)}{\delta- \alpha + \alpha \delta},
\end{align*}
and the symmetric expression $G_0(y,x)$ for $0\le y\le x\le 1$. 
\begin{example} \label{rk:3.2}
	One of the simplest cases of the Robin BCs is:
	$u(0)=u'(1)=0$, i.e., $\alpha \to \infty$ and $\delta \to 0$, $\Delta\to -\infty$, corresponding to Dirichlet BC at one end, and free Neumann BC at the other end.
	For $0\le x\le y\le 1$, the resolvent kernel is explicitly given by 
	\begin{align*} 
	G(z,x,y)=
	\frac{\sin(\sqrt{z} x) \cos\(\sqrt{z} (1-y)\)}
	{\sqrt{z} \cos\sqrt{z}} .
	\end{align*}
	The Green's function corresponding to this problem is then  
	\begin{equation*}
	G_0(x,y) = \lim_{z\to 0^{+}} \, G(z,x,y) =x  \quad \text{for $0\le x\le y\le 1$}, 
	\end{equation*}
	and 
	the symmetric expression $G_0(y,x)$ for $0 \leq y \leq x \leq 1$.
	\qed
\end{example}


\begin{example}\label{rk:3.3} 
	Another example appears in Kato's work as Example~4.14~\cite[p.~293]{KATO} where the BCs are given by $u(0)=0$ and $\tau u'(1)- u(1)=0$, and $0<\tau <1$ is a fixed parameter (our $\tau$ is the $\kappa$ in Kato's book).  Note that $\alpha \to \infty$ and $\delta =1/\tau$ in this example.
	
	Hence,
	\begin{equation*} 
	G(z,x,y)=
	\frac{\sin\(\sqrt{z} x\)}{\tau \, z \cos\sqrt{z} - \sqrt{z} \sin\sqrt{z}}   \, \big( \tau \, \sqrt{z} \, \cos(\sqrt{z} (1-y))  - \sin\(\sqrt{z} (1-y)\)\big).
	\end{equation*}
	This expression agrees with the result in Kato's book. 
	For $0\le x\le y\le 1$, the Green's function corresponding to this problem 
	\begin{equation*}
	G_0(x,y) = \lim_{z\to 0^{+}} \, G(z,x,y) = x - \frac{1}{1 - \tau} xy 
	\end{equation*}
	and the symmetric expression for $0\le y\le x\le 1$. We note that this problem admits a unique negative eigenvalue; see \cite{KATO}. 
	
	\qed
\end{example}

\begin{example} \label{rk:3.4}
	Another variation of this Robin problem appears as Example~1.4~\cite[p.~367]{KATO} where the BCs are replaced with $u(0)=0$ and $u'(1)+ \tau u(1)=0$, i.e., $\alpha \to \infty$ and $\delta=-1/\tau$ in the general setting. Hence, we have
	\begin{align*} 
	G(z,x,y)=
	\frac{\sin\(\sqrt{z} x\)  \, \(\sqrt{z}\, \cos\(\sqrt{z} (1-y)\) + \tau \,  \sin \(\sqrt{z} (1-y)\) \)}{z \cos\sqrt{z} + \tau \, \sqrt{z} \sin\sqrt{z} }.
	\end{align*}
	for $0\le x\le y\le 1$. This expression is also equivalent to Kato \cite{KATO}. The Green's function corresponding to this problem is given by
	\begin{equation} \label{eq:integral-Kato4.14}
	G_0(x,y) = \lim_{z\to 0^{+}} \, G(z,x,y) = x - \frac{\tau}{1+\tau}x y \quad \text{for $0\le x\le y\le 1$}.
	\end{equation}
	Unlike \cref{rk:3.3}, this problem does not exhibit negative eigenvalues.
	\qed
\end{example}
\begin{example} \label{rk:3.5}
	Yet another variation of this Robin eigenvalue problem also appears in the work of Stakgold and Holst \cite[Example 1, pp.416--420]{StakgoldGreenBook}, with BCs of the form: $u(0) = 0$, and $u'(1)\sin \theta + u(1) \cos \theta =0$. 
	This is equivalent to 
	\cref{rk:3.4} with $\tau= \cot \theta$ (or $\alpha \to \infty$ and $\delta=-\tan \theta$). 
	
	For $0\le x\le y\le 1$, and in terms of the angle $\theta$, the resolvent kernel takes the form
	\begin{align*} 
	G(z,x,y)=
	\frac{\sin\(\sqrt{z} x\) \, \(\sqrt{z}\,  \sin\theta \, \cos\(\sqrt{z} (1-y)\) 
		+   \cos \theta  \sin \(\sqrt{z} (1-y)\)\)}{z\, \sin\theta \, \cos \sqrt{z} + \sqrt{z} \cos \theta \, \sin \sqrt{z}} 
	\end{align*}
	while, the Green's function corresponding to the integral operator is then given by
	\begin{equation} \label{eq:integral-Stakgold}
	G_0(x,y):=\lim_{z\to 0^{+}} \, G(z,x,y) = x - \frac{1}{1+\tan\theta}x y.
	\end{equation}
	These expressions are equivalent to those appearing in \cite{StakgoldGreenBook}. 
	\qed
\end{example}

\begin{remark} \label{rk:III.6} For $\delta-\alpha+ \alpha \delta =0$, i.e., $\delta=\frac{\alpha}{\alpha+1}$ or $\alpha = \frac{\delta}{1-\delta}$, the Riesz projection $P$ has a kernel given by
\begin{align*}
p(x,y)=-\frac{3 \left(1 + \alpha x \right) \, \left(1 + \alpha y \right)}{3 +3 \alpha +\alpha^2}
\end{align*}
while the nilpotent operator $D\equiv 0$. 
The Green's function corresponding to this problem is then  
\begin{align*}
G_0(x,y) &= \lim_{z\to 0^{+}} \, G(z,x,y) - \frac{p(x,y)}{z} \notag\\ 
&=  
c_{00} + c_{10} x + c_{01} y + c_{20} x^2 + c_{11} x y + c_{02} y^2 + c_{30} x^3 + c_{21} x^2 y
\\
& 	\hspace{1em} + c_{12} x y^2 + c_{03} y^3 +c_{31} x^3 y+ c_{13} x y^3
\end{align*}
for $0\le x\le y\le 1$, where the coefficients are explicitly given in \cref{Appendix:A}, with $\beta=0$. The expression $3 +3 \alpha +\alpha^2>0$ (its discriminant is negative). It appears in the expression of the Riesz projection kernel, and in all the expressions of the $c_{ij}$ coefficients. The Green's function $G_0(x,y)$ is a fourth-order polynomial in $x$ and $y$. 
\end{remark}

\subsection{Periodic Boundary Conditions}
\label{sec:Periodic Resolvent}
Historically, this is a well-studied eigenvalue pro-blem, due to its tight connection to Fourier series \cite{Freeden2011}. Although it is not of GSARC type, we include it for comparison sake with the work of \cite{FGKLNS2021}. In lieu of the BCs in \cref{eq:general}, we treat the eigenvalue problem with $u(0)=u(1)$ and $u'(0)=u'(1)$.  

Solving the Helmholtz equation \cref{eq:Helm} for $u$ with these BCs, and imposing the continuity and jump conditions as before, leads, 
for $0\le x\le y\le 1$, to 	
\begin{align*} 
G(z,x,y)=-\frac{\cos\(\frac{\sqrt{z}}{2}(1+2x-2y)\)}{2\sqrt{z} \sin\frac{\sqrt{z}}{2}}. 
\end{align*}
The kernel of Riesz projection $P$ is given by $p(x,y)=\lim_{z\to 0^{+}} z \, G(z, x,y)=-1$. 
The nilpotent operator $D\equiv 0$ since \\ $\lim_{z\to 0^{+}} z^2 \, G(z, x,y)=0$. 

The Green's function corresponding to this periodic case is then  
\begin{align*}
G_0(x,y) &=\lim_{z\to 0^{+}} \, G(z,x,y) - \frac{p(x,y)}{z} \notag\\ 
&= \frac{1}{12} 
\(1+ 6 x - 6 y + 6 x^2 - 12 x y + 6 y^2\),
\end{align*}
for $0\le x\le y\le 1$, and the corresponding expression by symmetry for $0\le y\le x\le 1$.

\subsection{Anti-Periodic Boundary Conditions}
\label{sec:Anti-Periodic Resolvent}
This is one of the eigenvalue problems revisited in the recent article of Fucci, \textit{et al.} \cite{FGKLNS2021}, though not an GSARC problem. The BCs are given by  
$u(0)=- u(1)$ and $u'(0)=-u'(1)$. The resolvent kernel ensuing from the proposed technique herein is given, for $0\le x\le y\le 1$, by 	
\begin{align*} 
G(z,x,y)=\frac{\sin\(\frac{\sqrt{z}}{2} (1+2x-2y)\)}{2 \sqrt{z} \cos\frac{\sqrt{z}}{2}}		
\end{align*}
The kernel of Riesz projection $P\equiv 0$ since $p(x,y)=\lim_{z\to 0^{+}} z \, G(z, x,y)=0$, and thus the nilpotent operator $D\equiv 0$.

The Green's function corresponding to the anti-periodic case is then  
\begin{align*}
G_0(x,y)=\lim_{z\to 0^{+}} \, G(z,x,y) = \frac{1}{4} \(1+ 2 x - 2y\)
\end{align*}
for $0\le x\le y\le 1$, and the corresponding expression by symmetry for $0\le y\le x\le 1$.

\subsection{Radoux Boundary Conditions}
\label{sec:Radoux Resolvent}
The Radoux eigenvalue problem \cite{RADOUX} appears in the literature in the context of obtaining sum rules for the roots for the transcendental equation $\tan x = x$. Sum rules are closed form expressions for the zeta function of these roots with a history that is traceable back to Rayleigh and others; see \cite{HS-ACHA18}. Radoux used spectral techniques for solving 
\cref{eq:general} with the BCs $u(0)=0$, $u'(1)=u(1)$ to find these sum rules. This eigenvalue problem is a limiting case of \cref{sec:Robin Resolvent} 
when $\alpha \to \infty$, and $\delta =1$, (or when $\tau \to 1$ in \cref{rk:3.3}). The discriminant $\Delta \to \infty$ and the formulas of \cref{sec:Robin Resolvent} are no longer valid. 

While the resolvent kernel is recoverable from this limit, the Green's function is not. Proceeding as before, we obtain  
\begin{align*} 
G(z,x,y)=
\frac{\sin\(\sqrt{z} x\) \, \(\sqrt{z}\, \cos\(\sqrt{z} (1-y)\) - \sin\(\sqrt{z} (1-y)\)\)}{z \cos\sqrt{z} - \sqrt{z} \sin\sqrt{z} }  .
\end{align*}
The kernel of Riesz projection $P$ is given by $p(x,y)=\lim_{z\to 0^{+}} z \, G(z, x,y)=- 3 x \, y $, and the nilpotent operator $D\equiv 0$ since $\lim_{z\to 0^{+}} z^2 \, G(z, x,y)=0$. For $0\le x\le y\le 1$, the Green's function is then  
\begin{equation} \label{eq:integral-Radoux}
G_0(x,y) = \lim_{z\to 0^{+}} \, G(z,x,y) -\frac{p(x,y)}{z} = \frac{1}{10} \, \(- 10 x+ 18 xy - 5 x^3 y - 5 x y^3\). 
\end{equation}

\subsection{GSARC Boundary Conditions with Non-Zero Discriminant} \label{sec:GSARC with Non-Zero Discriminant Resolvent}

We now focus on the general problem \cref{eq:general}, with $\Delta=\delta-\alpha- 2 \beta + \beta^2 + \alpha \delta \neq 0$. The technique described yields a resolvent kernel where 

The expression for 
\begin{align*}
\(-4 \beta z + 2z \(-\alpha+ \delta\)\cos\sqrt{z} + 2 \sqrt{z} \(z+ \beta^2+ \alpha \delta\)\sin\sqrt{z}\) \, G(z,x,y)
\end{align*} 
is given by 
\begin{align*}
& -\(z+ \beta^2 + \alpha \delta\) \cos\(\sqrt{z}(1+x-y)\) + \(-z+ \beta^2 + \alpha \delta\) \cos\(\sqrt{z}(1-x-y)\) \\
& \hspace{1em} + \sqrt{z} \big(- 2\beta \sin\(\sqrt{z}(x-y)\)
+ \(-\alpha + \delta\) \sin\(\sqrt{z}(1+x-y)\) \\
& \hspace{2em} + \(\alpha + \delta\) \sin\(\sqrt{z}(1-x-y)\) \big)
\end{align*} 
for $0\le x\le y\le 1$, and thus by symmetry for full expression in  \cref{tab:Resolvent}. For $0\le x\le y\le 1$, the kernel of Riesz projection $P \equiv 0 \equiv D$ since $\lim_{z\to 0^{+}} z^n \, G(z, x,y)=0$, for $n=1,2$. The Green's function is then
\begin{align*}
G_0(x,y) &= \lim_{z\to 0^{+}} \, G(z, x,y) \\ 
&= \frac{-1+\delta - \(\alpha + \beta\) x + \(\beta-\delta\) y+ \(\beta^2+ \alpha \delta \) x (1-y)}{\delta-\alpha- 2 \beta + \beta^2 + \alpha \delta}
\end{align*} 
for $0\le x \le y \le 1$, and the symmetric expression $G_0(y,x)$ for  $0\le  y \le x \le 1$.


\subsection{GSARC Boundary Conditions with Zero Discriminant}
\label{sec:GSARC with Zero Discriminant  Resolvent}
We now discuss the case when the discriminant $\Delta=\delta-\alpha- 2 \beta + \beta^2 + \alpha \delta = 0$. We first discuss the generic case, with
$\delta=\frac{\alpha + 2 \beta - \beta^2}{\alpha +1}$ (equivalently
$\alpha=\frac{\delta - 2 \beta + \beta^2}{1-\delta}$), then discuss the three limiting cases: (i)
$\alpha \to -1$ (so that $\beta=1$, and $\delta$ is arbitrary); (ii) $\delta \to 1$ (so that $\beta=1$, and $\alpha$ is arbitrary); and (iii) the special case satisfying both (i) and (ii), i.e., $\alpha = -1$, $\beta = 1$, $\delta=1$. 

The expression of the resolvent kernel following from the same method is such that 
\begin{align*} 
2 z \(-2 (1+ \alpha) \beta -  \, (\alpha^2+ \beta^2 - 2 \beta)\, \cos\sqrt{z} +  \, \((1+ \alpha) z + (\alpha+ \beta)^2\) \sin\sqrt{z}
\)	G(z,x,y)
\end{align*}
is equal to 
\begin{hide}
	\begin{align*} 
	& & - \( (1+\alpha) z+ (\alpha + \beta)^2 \) \cos\( \sqrt{z} (1+x-y)\)	+\(-(1+\alpha) z +(\alpha + \beta)^2\) \, \cos \(\sqrt{z} (1-x-y)\) \\
	& &	\hspace{-1em} + \sqrt{z} \(-2 (1+\alpha) \, \beta \sin \(\sqrt{z}(x-y) \) 
	-(\alpha^2+\beta^2 - 2 \beta) \sin \(\sqrt{z} (1+x-y)\) + (2+\alpha - \beta) (\alpha+\beta) \sin \(\sqrt{z} (1-x-y)\) 
	\)
	\end{align*}
\end{hide}
\begin{align*}
& - \( (1+\alpha) z+ (\alpha + \beta)^2 \) \cos\( \sqrt{z} (1+x-y)\)	+\(-(1+\alpha) z +(\alpha + \beta)^2\) \, \cos \(\sqrt{z} (1-x-y)\) \\
& \hspace{1em}  + \sqrt{z} \bigg\{ -2 (1+\alpha) \, \beta \sin \(\sqrt{z}(x-y) \) 
-(\alpha^2+\beta^2 - 2 \beta) \sin \(\sqrt{z} (1+x-y)\) \\
& \hspace{2em} + (2+\alpha - \beta) (\alpha+\beta) \sin \(\sqrt{z} (1-x-y)\)  \bigg\}
\end{align*}
Upon calculating $\lim_{z\to 0^{+}} z^n \, G(z, x,y)$ for $n=1, 2,\ldots$, we realize that the kernel of the Riesz projection is given by 
\begin{align*}
p(x,y)=-\frac{3 \left(1-\beta + \left(\alpha+\beta \right) x \right) \, \left(1-\beta + \left(\alpha+\beta \right) y \right)}{3+\alpha^2 + \beta^2 - \alpha \beta +3 \alpha - 3 \beta}
\end{align*}
while the nilpotent operator $D\equiv 0$. 

The Green's function corresponding to this problem is then  
\begin{align*}
G_0(x,y) &= \lim_{z\to 0^{+}} \, G(z,x,y) - \frac{p(x,y)}{z} \notag\\ 
&=  
c_{00} + c_{10} x + c_{01} y + c_{20} x^2 + c_{11} x y + c_{02} y^2 + c_{30} x^3 + c_{21} x^2 y
\\
& 	\hspace{1em} + c_{12} x y^2 + c_{03} y^3 +c_{31} x^3 y+ c_{13} x y^3
\end{align*}
for $0\le x\le y\le 1$. 
The coefficients depend on the parameters and are explicitly given in \cref{Appendix:A}.

We opted not to include the expressions for $G(z,x,y)$ and $G_0(x,y)$ in \cref{tab:Resolvent} and \cref{tab:Green} since the expressions are too long. 

\begin{remark} \label{rk3.4}
	Note that $c_{10}- c_{01}=1$, $c_{20}=c_{02}$, $c_{21}=c_{12}$, $c_{30}=c_{03}$, and $c_{31}=c_{13}$. These identities will prove useful in the calculation of the perturbation in \cref{sec:Perturb GSARC with Non-Zero Discriminant}.
	\qed
\end{remark}

\begin{remark} \label{rk:III.8}
	
	We now treat the limiting cases (i), (ii), and (iii). Again, we list the results for $0\le x\le y\le 1$ and complete the expression for $0\le y\le x\le 1$ by symmetry.
	\begin{itemize}
		\item[(i)] $\alpha \to -1$ (so that $\beta=1$, and $\delta$ is arbitrary)
		
		The expression of the resolvent kernel following from the same method is such that 
		\begin{align*} 
		\(\sqrt{z} \(-2 + (1+ \delta) \cos\sqrt{z}\)
		+ (1 -\delta+z) \sin\sqrt{z}\)	G(z,x,y)
		\end{align*}
		is equal to 
		\begin{align*} 
		& & \(-\sqrt{z} \, \cos\( \sqrt{z} (1-y)\)	+ \delta \, \sin\(\sqrt{z} (1-y)\) + \sin\(\sqrt{z} y\)\) \, \cos \(\sqrt{z} x \)
		\\
		& &	\hspace{1em} + \(\sqrt{z} \, \cos\(\sqrt{z}(1-y)\) - \cos\(\sqrt{z} y\) +(1-\delta) \sin \(\sqrt{z}(1-y)\)\) 
		\, \frac{\sin\(\sqrt{z} x\)}{\sqrt{z}} .
		\end{align*}
		The Riesz projection has a kernel
		\begin{align*}
		p(x,y)=-3 (1-x) \, (1-y)
		\end{align*}	
		and the corresponding Green's function is
		\begin{align*} 
		G_0(x,y) &=-\frac{1}{20 (1-\delta)} \, 
		\big((-9+ 4 \delta) + (19- 4 \delta) x - 
		30 (1-\delta) x^2 + 10 (1-\delta) x^3
		\\
		& 	\hspace{5em} +(39 - 24 \delta) y +(-69 + 24 \delta) y^2 + 10(1-\delta) y^3 \\
		& 	\hspace{5em} +3 (-23+ 8 \delta) x y + 30 (1-\delta) x y^2 + 30 (1-\delta) x^2 y
		\\
		& 	\hspace{5em} -10 (1-\delta) x^3 y - 10 (1-\delta) x y^3\big) .
		\end{align*}
		\item[(ii)] $\delta \to 1$ (so that $\beta=1$, and $\alpha$ is arbitrary). 
		The expression of the resolvent kernel following from the same method is such that 
		\begin{align*} 
		\(\sqrt{z} \(-2 + (1- \alpha) \cos\sqrt{z}\)
		+ (1+\alpha + z) \sin\sqrt{z}\)	G(z,x,y)
		\end{align*}
		is equal to 
		\begin{align*} 
		&  \(-\sqrt{z} \, \cos\(\sqrt{z} (1-y)\) +  \sin\(\sqrt{z} (1-y)\) + \sin\(\sqrt{z} y\)\) \, \cos\(\sqrt{z} x\)
		\\
		& 	\hspace{1em} - \sqrt{z} \, \(\alpha \, \cos\(\sqrt{z}(1-y)) + \cos\(\sqrt{z} y \)\) -(1+\alpha) \sin\(\sqrt{z}(1-y)\)\) 
		\, \frac{\sin\(\sqrt{z} x\)}{\sqrt{z}} .
		\end{align*}
		The Riesz projection has a kernel
		\begin{align*}
		p(x,y)=- 3 x y 
		\end{align*}	
		and the corresponding Green's function is
		\begin{align*} 
		G_0(x,y) &=\frac{1}{20 (1+\alpha)} \, 
		\big(20 - 30 x + 10 (-1 + 2 \alpha) y 
		\\
		& 	\hspace{5em} + (9- 36 \alpha ) x y  + 10 (1+\alpha) x y^3 + 10 (1+\alpha) x^3 y\big) .
		\end{align*}
		
		\item[(iii)] $\alpha=-1$, $\beta=1$, and $\delta = 1$ 
		
		This is the most particular of all these cases.
		The expression of the resolvent kernel following from the same method is such that 
		\begin{align*} 
		\(2 \sqrt{z} (-1+\cos\sqrt{z} + z \, \sin\sqrt{z}\) G(z,x,y)
		\end{align*}
		is equal to 
		\begin{align*} 
		&  \(-\sqrt{z} \, \cos\(\sqrt{z} (1-y)\) +  \sin\(\sqrt{z} (1-y)\) + \sin\(\sqrt{z} y\)\) \, \cos \(\sqrt{z} x \)
		\\
		& \hspace{1em} -\sin \(\sqrt{z} \(\frac{1}{2}-y\) \)\,  
		\, \frac{\sin\(\sqrt{z} x\)}{2\sin\frac{\sqrt{z}}{2}} .
		\end{align*}
		The Riesz projection kernel is given by $p(x,y) = -4 + 6 x + 6 y - 12 x y$.
		The corresponding Green's function is
		\begin{align*} 
		G_0(x,y) &= \frac{1}{30}\big(4-3x + 60 x^2 - 30 x^3 - 33y + 36 x y - 90x^2 y\\ 
		&  \hspace{1em} +60 x^3 y+ 60 y^2- 90 xy^2- 30 y^3 + 60 x y^3 \big).
		\end{align*}
	\end{itemize} 
	In all cases, the nilpotent operator $D\equiv 0$. 
	\qed
\end{remark}

\section{Pertubations of the Nonlocal Operator Eq.~(I.1)}
\label{sec:perturbation}
With $\mathcal{K}$ as defined in \cref{eq:Saito-Operator}, we are interested in understanding the nature of the integral operator $\mathcal{T}:=\mathcal{K}-\mathcal{K_G}$, where
$\mathcal{K_G} f(x):=\int_0^1 G_0(x,y) f(y) \dd{y}$, $G_0(x,y)$ is one of the Green's functions developed in \cref{sec:resolvent-integral} for various GSARC problems. The perturbation $\mathcal{T}$ we aim to study is then given by
\begin{equation} \label{eq:gen-perturb}
\mathcal{T} f (x) := \int_0^1 \left(\kappa(x,y)- G_0(x,y)\right) f(y) \dd{y}, 
\end{equation}
where $\kappa(x,y)$ was defined by \cref{eq:kappa}.
The method consists of finding the eigenvalues and corresponding eigenfunctions of this operator using Linear Algebra. In each of the BVPs in this article, we proceed as follows: 
\begin{description}
	\item[Step 1] Determine the form of the kernel $T(x,y)= \kappa(x,y)- G_0(x,y)$ for $x\le y$;
	\item[Step 2] Set the kernel $T(x,y)$, for $x\ge y$, by symmetry;
	\item[Step 3] Simplify the expression of 
	\begin{equation} \label{eq:Txy}
	\int_0^1 T(x,y) f(y) \dd{y}= \int_0^x T(x,y) f(y) \dd{y} + \int_x^1 T(y,x) f(y) \dd{y};
	\end{equation}
	(This leads to the appropriate form of the function $f(x)$. 
	It turns out to be a polynomial with a degree off by 1 from the size of the matrix corresponding to the operator $\mathcal{T}$.)
	\item[Step 4] 
	Find the spectral resolution of the operator $\mathcal{T}$ by working out the spectral resolution of the corresponding matrix problem.
\end{description} 
The range of the operator $\mathcal{T}$  is a space of polynomials. If $n$ is the rank of the matrix $M_{\mathcal{T}}$ corresponding to the operator $\mathcal{T}$. 
The spectral resolution of this operator is given by
\begin{equation}
\mathcal{T} f (x) = \sum_{k=1}^n \langle f, u_k \rangle \, u_k(x)
\end{equation}
where $\langle f, g \rangle \define \int_0^1 f(x) g(x) \dd{x}$ is the usual dot product, 
and $u_1, \ldots, u_n$ are the eigenvectors of the matrix $M_{\mathcal{T}}$ 
which we determine for the various eigenvalue problems discussed earlier.

In all cases, we will show that $T(x,y)=T(y,x)$, thus the split of the integral in \cref{eq:Txy} is not needed. In fact the resolved kernel associated with $\mathcal{T}$ is symmetric as well in $x$ and $y$.

\begin{remark}
	Perturbation is treated in a multiplicative framework in Section 2.8 of  \cite{BehrndtHassideSnoo2020} where the finiteness of the rank of the perturbation implies the existence of boundary triplets (see Theorem 2.8.1). 
	Kre\u{\i}n-type resolvent formulas (and thus Green functions) appear in
	\cite[Sec.4.3, 7.5, 11.3, 13.10, 14.13, Appendix D.6]{GNZ-AMSBook2023}.
\end{remark}

\subsection{Dirichlet Boundary Conditions}
\label{sec:Dirichlet Problem}
As noted in \cite{SAITO-LAPEIG-ACHA} and \cite{HS-ACHA18}, the Green's function corresponding to the Dirichlet problem (consult \cref{tab:Green}) is given by 
\begin{equation} \label{eq:G0-Dirichlet}
G_0(x,y)=\frac{1}{2} \left(x+y\right) - x y -\frac{1}{2} |x-y|= -\left(x-\frac{1}{2}\right) \left(y-\frac{1}{2}\right)+\frac{1}{4}-\frac{1}{2} |x-y|.
\end{equation}
The latter form appears in the 1959 paper of Lidski\u{\i} \cite{Lidskii59}. From \cref{eq:G0-Dirichlet}, we easily get 
\begin{align*}
T(x,y)= x y -\frac{1}{2} (x+y).
\end{align*}
Since this expression is symmetric, viz. $T(x,y)=T(y,x)$, the perturbation operator takes the simple expression 
\begin{align} \label{eq:Dperturb}
\mathcal{T} f (x) &= \int_{0}^1 \left(xy - \frac{1}{2} \left(x+y\right) \right) f(y) \dd{y} \notag \\
&= -\left(\int_{0}^1 \frac{1}{2} y f(y) \dd{y} \right) - \left(\int_0^1 (\frac{1}{2} - y ) f(y) \dd{y}\right) x.
\end{align}
By virtue of \cref{eq:Dperturb}, the eigenvalue problem $\mathcal{T}f(x) = \lambda f(x)$
leads to a linear eigenfunctions of the form $f(x) = a_0+a_1 x$ and the matrix eigenvalue problem
\[\begin{pmatrix}
-\frac{1}{4} & - \frac{1}{6} \\
0 & \frac{1}{12}
\end{pmatrix} \,  \begin{pmatrix}
a_0  \\
a_1 \end{pmatrix} = \lambda \begin{pmatrix}
a_0  \\
a_1 
\end{pmatrix} . \]
Thus $\lambda_1=-\frac{1}{4}$, $\lambda_2=\frac{1}{12}$, and the corresponding eigenvectors are
\begin{equation}\label{eq:pDevs}
u_1(x)=1, \quad u_2(x)=x-\frac{1}{2}.
\end{equation}
This agrees with Lidski\u{\i} \cite{Lidskii59} and our earlier work \cite{HS-ACHA18}. We have the immediate theorem:
\begin{theorem}\label{thm:dirichlet}
	The Dirichlet operator $\mathcal{K_{G}}$ is a rank 2 perturbation of $\mathcal{K}$. Moreover, \begin{equation*}
	\mathcal{T} f (x) = \langle f, u_1\rangle u_1(x) + \langle f, u_2\rangle  u_2(x), 
	\end{equation*}
	where $u_1(x), u_2(x)$ are defined in \cref{eq:pDevs}.
\end{theorem}

\subsection{Neumann Boundary Conditions}
\label{sec:Neumann Problem}
The calculation of the perturbation $\mathcal{T}$ appears as Remark~2.9 in \cite{HS-ACHA18}, where we used the method of \cite{AbramovichAliprantis1, AbramovichAliprantis2}. We offer here a succinct scheme. From the expression of Green's function, 
\begin{equation*}
G_0(x,y)=\frac{1}{6} \left(2+ 3 x^2- 6 y+ 3y^2\right) ,
\end{equation*}
the kernel of the perturbation is 
\begin{equation*}
T(x,y)=-\frac{1}{6} \left(2 -3 x+ 3 x^2 - 3 y+ 3 y^2\right)
\end{equation*}
for $x\le y$. Since this expression is symmetric, viz. $T(x,y)=T(y,x)$, the operator $\mathcal{T}$ acts as
\begin{align}\label{eq:Nperturb}
\mathcal{T} f (x) &= 
\int_{0}^1 T(x,y) f(y) \dd{y} \notag \\
&= \left(\int_{0}^1 -\frac{1}{6} \left(2- 3 y + 3 y^2\right) f(y) \dd{y} \right)+ \left(\int_0^1 \frac{1}{2} f(y) \dd{y}\right) \left(x - x^2\right) .
\end{align}
Hence, the eigenvalue problem corresponding to this operator, $\mathcal{T} f(x) = \lambda f(x)$, leads to a quadratic form of the eigenvector $f(x)= a_0+ a_1 \left(x - x^2\right)$, and the matrix eigenvalue problem
\begin{hide}
	\[\begin{pmatrix}
	-\frac{1}{4} & - \frac{1}{8} &  - \frac{31}{360}\\ \\
	\frac{1}{2} & \frac{1}{4} &  \frac{1}{6}\\ \\
	-\frac{1}{2} & - \frac{1}{4} &  - \frac{1}{6}
	\end{pmatrix} \,  \begin{pmatrix}
	a_0  \\
	a_1 \\
	a_2 
	\end{pmatrix} = \lambda \begin{pmatrix}
	a_0  \\
	a_1\\
	a_2 
	\end{pmatrix} .\]
	Thus $\lambda_1=0$, $\lambda_2=\frac{1}{60} \left(-5-\sqrt{30}\right)$, and 
	$\lambda_3=\frac{1}{60} \left(-5+\sqrt{30}\right)$. The corresponding eigenvectors are:
	\begin{equation}\label{eq:pNevs}
	u_1(x)=-\frac{1}{2} + x, \quad u_2(x) =\frac{1}{30} \left(10+\sqrt{30}\right) - x+x^2, \quad 
	u_3(x)=\frac{1}{30} \left(10-\sqrt{30}\right) - x+x^2. 
	\end{equation}
	Hence we have:
	\begin{theorem}\label{thm:neumann}
		The Neumann operator $\mathcal{K_{G}}$ is a rank 3 perturbation of $\mathcal{K}$. Moreover, \begin{equation*}
		\mathcal{T} f (x) = \langle f, u_1\rangle u_1(x) + \langle f, u_2\rangle  u_2(x)+ \langle f, u_3\rangle u_3(x),
		\end{equation*}
		where $u_1(x), u_2(x), u_3(x)$ are defined in \cref{eq:pNevs}.
	\end{theorem}
\end{hide}
\[\begin{pmatrix}
-\frac{1}{4} & - \frac{7}{180}\\ \\  
\frac{1}{2} & \frac{1}{12}
\end{pmatrix} \,  \begin{pmatrix}
a_0  \\
a_1
\end{pmatrix} = \lambda \begin{pmatrix}
a_0  \\
a_1
\end{pmatrix} .\]
Thus $\lambda_1=\frac{1}{60} \left(-5-\sqrt{30}\right)$, and 
$\lambda_2=\frac{1}{60} \left(-5+\sqrt{30}\right)$. The corresponding eigenvectors are:
\begin{equation}\label{eq:pNevs}
u_1(x) =-\frac{1}{30} \left(10+\sqrt{30}\right) + x - x^2, \quad 
u_2(x)=\frac{1}{30} \left(-10+\sqrt{30}\right) + x- x^2. 
\end{equation}
Hence we have:
\begin{theorem}\label{thm:neumann}
	The Neumann operator $\mathcal{K_{G}}$ is a rank 2 perturbation of $\mathcal{K}$. Moreover, \begin{equation*}
	\mathcal{T} f (x) = \langle f, u_1\rangle u_1(x) + \langle f, u_2\rangle  u_2(x),
	\end{equation*}
	where $u_1(x)$ and $u_2(x)$ are defined in \cref{eq:pNevs}.
\end{theorem}

\begin{remark}
The expression of the perturbation operator $\mathcal{T} f (x)$ in \cref{eq:Nperturb} is quadratic in $x$, and thus one can simply write $f(x)= a_0+ a_1 x + a_2 x^2$. This leads to an eigenvalue problem with a matrix $M_{\mathcal{T}}$ given by 
\begin{equation*}
M_{\mathcal{T}}= \begin{pmatrix}
-\frac{1}{4} & -\frac{1}{8} &  -\frac{31}{361}  \\ \\
\frac{1}{2} & \frac{1}{4} &  \frac{1}{6}  \\ \\
-\frac{1}{2} & -\frac{1}{4} &  -\frac{1}{6}  \\ \\
\end{pmatrix}.
\end{equation*}
Notice that the first column is twice the second, hence the column vectors are linearly dependent and the rank of this $3\times 3$ matrix is 2. Indeed the eigenvalues of this matrix are $\lambda_1$ and $\lambda_2$ above, with the eigenfunctions given in \cref{eq:pNevs}, and 
$\lambda_0=0$, with corresponding eigenfunction $u_0(x)=-\frac{1}{2}+x$. Alternatively one can perform a Singular Value Decomposition (SVD) of the matrix $M_{\mathcal{T}}$ which immediately gives the number of non-zero singular values in this case as 2. This is also the rank since the number of non-zero singular values is also the rank (see 
Theorem 5.1 of Ref.~\onlinecite{TREFETHEN-BAU}). 
\end{remark}

\subsection{Kre\u{\i}n-von Neumann Boundary Conditions}
\label{sec:KvN Problem}
In \cref{sec:KvN Resolvent}, we derived the Green's function $G_0(x,y)$ corresponding to this problem with BCs $u'(0)=u'(1)=-u(0)+u(1)$. For $0\le x\le y\le 1$, it is explicitly given by 
\begin{align*} \label{eq:integral-KvN}
G_0(x,y) &= \frac{1}{30} \big(
4 - 3 x - 33 y+60 x^2+ 60 y^2 + 36 xy \\
& \hspace{5em} -30 x^3-90 x^2 y - 90 x y^2- 30 y^3+ 60 x^3 y+ 60 x y^3\big). 
\end{align*}
The expression of $T(x,y)$ is given by 
\begin{align*} 
T(x,y) &=
\frac{1}{15} \big(-2+ 9 x - 30 x^2 + 15 x^3 + 45 x^2 y - 30 x^3 y \\ 
& 	\hspace{5em}- 18 x y - 30 x^3 y + 45 x y^2 + 15 y^3 - 30 y^2 + 9 y \big).
\end{align*}
This is again symmetric, viz. $T(x,y)=T(y,x)$. This is a cubic polynomial in $x$, and thus the eigenvector of the matrix problem is of the form $f(x)=a_0+a_1 x+a_2 x^2 +a_3 x^3$. The corresponding matrix has the explicit form  
\begin{equation}
M_{\mathcal{T}}= \begin{pmatrix}
-\frac{1}{4} & - \frac{1}{6} &  - \frac{23}{180} &  - \frac{109}{1050}  \\ \\
\frac{1}{2} & \frac{1}{4} &  \frac{1}{6} &  \frac{87}{700} \\ \\
-\frac{1}{2} & 0 &   \frac{1}{12} &  \frac{1}{10} \\ \\ 
0 & - \frac{1}{6} &  - \frac{1}{6} &  - \frac{3}{20} \\ \\ 
\end{pmatrix}.
\end{equation}
Its eigenvalues are: $\lambda_1=\frac{1}{60} \left(-5 - \sqrt{30}\right)$, $\lambda_2=\frac{1}{60} \left(-5 + \sqrt{30}\right)$, $\lambda_3=\frac{1}{420} \left(21- \sqrt{462}\right)$, and $\lambda_4=\frac{1}{420} \left(21+ \sqrt{462}\right)$. The corresponding eigenvectors are
\begin{equation}\label{eq:pKvNevs}
\begin{aligned}
u_1(x)&=\frac{1}{30} \left(10+\sqrt{30}\right)- x+ x^2, \\
u_2(x)&=\frac{1}{30} \left(10-\sqrt{30}\right)- x+ x^2, \\
u_3(x)&=\frac{1}{140} \left(14-\sqrt{462}\right)- \frac{1}{70} \left(21+\sqrt{462}\right) x-\frac{3}{2} x^2 + x^3, \\
u_4(x)&=\frac{1}{140} \left(14+\sqrt{462}\right)- \frac{1}{70} \left(21-\sqrt{462}\right) x-\frac{3}{2} x^2 + x^3.
\end{aligned}
\end{equation}
We now have:
\begin{theorem}\label{thm:KvN}
	The Kre\u{\i}n-von Neumann operator $\mathcal{K_{G}}$ is a rank 4 perturbation of $\mathcal{K}$. Moreover, \begin{equation*}
	\mathcal{T} f (x) = \sum_{k=1}^4 \langle f, u_k\rangle u_k(x),
	\end{equation*}
	where $u_k(x)$, $k=1, 2, 3, 4$, are defined in \cref{eq:pKvNevs}.
\end{theorem}

\subsection{Robin Boundary Conditions}
\label{sec:Perturb Robin Problem}
Using the expression of the Green's function developed in \cref{sec:Robin Resolvent},
the expression of the perturbation is given by
\begin{equation*}
T(x,y)= \frac{2 - 2 \delta + \(\alpha+\delta - \alpha \delta\) x + \(\alpha+\delta - \alpha \delta\) y + 2 \alpha \delta x y}{2 \(\delta - \alpha + \alpha \delta\)}
\end{equation*}
for $0\le x \le y \le 1$, with $T(x,y)=T(y,x)$ for all $x, y \in [0,1]$. The eigenvectors are then of the form $f(x)= a_0+ a_1 x$. The matrix is given by, 
\begin{equation*}
M_{\mathcal{T}}= \begin{pmatrix}
\frac{4 + \alpha - 3 \delta- \alpha \delta}{4 \(\delta - \alpha + \alpha \delta\)} & \frac{3 + \alpha- 2 \delta - \alpha \delta}{6 \(\delta - \alpha + \alpha \delta\)}  \\ \\
\frac{\alpha+ \delta}{2\(\delta - \alpha + \alpha \delta\)} & 
\frac{3 \alpha+ 3 \delta + \alpha \delta}{12 \(\delta - \alpha + \alpha \delta\)}
\end{pmatrix}.
\end{equation*}
with eigenvalues \\
$\lambda_1=\frac{6+ 3 \alpha - 3 \delta - \alpha \delta - 2 
	\sqrt{\(3 + 3 \alpha + \alpha^2\) \(3 - 3 \delta + \delta^2\) }}{12\(\delta - \alpha + \alpha \delta\) }$,
$\lambda_2=\frac{6+ 3 \alpha - 3 \delta - \alpha \delta + 2 
	\sqrt{\(3 + 3 \alpha + \alpha^2\) \(3 - 3 \delta + \delta^2\) }}{12\(\delta - \alpha + \alpha \delta\) }$,\\
and corresponding eigenvectors
\begin{equation*}
u_1(x)= \frac{3 - 3 \delta - \alpha \delta - 
	\sqrt{\(3 + 3 \alpha + \alpha^2\) \(3 - 3 \delta + \delta^2\) }}{3 \(\alpha + \delta\) }+ x
\end{equation*}
and
\begin{equation*}
u_2(x)= \frac{3 - 3 \delta - \alpha \delta + 
	\sqrt{\(3 + 3 \alpha + \alpha^2\) \(3 - 3 \delta + \delta^2\) }}{3 \(\alpha + \delta\) }+ x. 
\end{equation*}

\begin{example}
	In \cref{rk:3.2} the limiting BC case $u(0)=u'(1)=0$ example of the Robin problem was discussed. The perturbation kernel takes the form
	\begin{equation*}
	T(x,y)=- \frac{1}{2} \(x+y\)
	\end{equation*} 
	for $0\le x \le y \le 1$, with $T(x,y)=T(y,x)$ for all $x, y \in [0,1]$. The eigenvectors are then of the form $f(x)= a_0+ a_1 x$. The matrix is given by, 
	\begin{equation*}
	M_{\mathcal{T}}= \begin{pmatrix}
	-\frac{1}{4} & - \frac{1}{6}  \\ \\
	-\frac{1}{2} & 
	- \frac{1}{4}
	\end{pmatrix}.
	\end{equation*}
	with eigenvalues $\lambda_1=- \frac{3+2 \sqrt{3}}{12}$ and $\lambda_2= -
	\frac{3- 2 \sqrt{3}}{12}$,
	and corresponding eigenvectors
	\begin{equation*}
	u_1(x)= 
	\frac{1}{\sqrt{3}}+ x
	\end{equation*}
	and
	\begin{equation*}
	u_2(x)= - \frac{1}{\sqrt{3}}+ x.
	\end{equation*}
\end{example}

\begin{example}
	In Kato's first variation on the Robin example discussed in \cref{rk:3.3}, the perturbation kernel takes the form
	\begin{equation*}
	T(x,y)= -\frac{1}{2} (x+y) + \frac{1}{1-\tau}xy
	\end{equation*}
	for $0\le x \le y \le 1$. Clearly, $T(x,y)=T(y,x)$ for all $x, y \in [0,1]$. The eigenvectors are then of the form $f(x)= a_0+ a_1 x$. The matrix of the finite rank operator is $2\times 2$, 
	\begin{equation*}
	M_{\mathcal{T}}= 
	\begin{pmatrix}
	-\frac{1}{4} & - \frac{1}{6}  \\ \\
	\frac{\tau}{2 \(1-\tau\)} & 
	\frac{1 + 3 \tau}{12 \(1-\tau\) }
	\end{pmatrix}.
	\end{equation*}
	with eigenvalues 
	$\lambda_1=- \frac{1- 3\tau - 2 \sqrt{1-3 \tau + 3 \tau^2}}{12 \(1-\tau\) }$ 
	and 
	$\lambda_2= - \frac{1- 3\tau + 2 \sqrt{1-3 \tau + 3 \tau^2}}{12 \(1-\tau\) }$,
	and corresponding eigenvectors
	\begin{equation*}
	u_1(x)= 
	- \frac{1 - \sqrt{1-3 \tau + 3 \tau^2}}{3 \tau } + x
	\end{equation*}
	and
	\begin{equation*}
	u_2(x)= 
	- \frac{1 + \sqrt{1-3 \tau + 3 \tau^2}}{3 \tau } + x.
	\end{equation*}
\end{example}

\begin{example}
	In Kato's second variation on the Robin BC of \cref{rk:3.4}, the perturbation kernel takes the form
	\begin{equation*}
	T(x,y)= -\frac{1}{2} (x+y) + \frac{\tau}{1+\tau}xy
	\end{equation*} 
	for $0\le x \le y \le 1$. We also have $T(x,y)=T(y,x)$ for all $x, y \in [0,1]$. The eigenvectors are then of the form $f(x)= a_0+ a_1 x$. The associated matrix is also $2\times 2$, 
	\begin{equation*}
	M_{\mathcal{T}}= \begin{pmatrix}
	-\frac{1}{4} & - \frac{1}{6}  \\ \\
	- \frac{1}{2 \(1+\tau\)} & 
	-\frac{3 - \tau}{12 \(1+\tau\) }
	\end{pmatrix}.
	\end{equation*}
	with eigenvalues $\lambda_1=- \frac{3+ \tau + 2 \sqrt{3+3 \tau + \tau^2}}{12 \(1+\tau\) }$ and $\lambda_2= - \frac{3+\tau - 2 \sqrt{3+3 \tau + \tau^2}}{12 \(1+\tau\) }$,
	and corresponding eigenvectors
	\begin{equation*}
	u_1(x)= \frac{\tau + \sqrt{3+3 \tau + \tau^2}}{3} + x
	\end{equation*}
	and
	\begin{equation*}
	u_2(x)= \frac{\tau - \sqrt{1+3 \tau + \tau^2}}{3} + x.
	\end{equation*}
\end{example}

\begin{example}
	In \cref{rk:3.5} of Stakgold and Holst, the perturbation kernel takes the form
	\begin{equation*}
	T(x,y)= -\frac{1}{2} (x+y) + \frac{\cot \theta }{1 + \cot \theta} xy
	\end{equation*} 
	for $0\le x \le y \le 1$, and again $T(x,y)=T(y,x)$ for all $x, y \in [0,1]$. The eigenvectors are then of the form $f(x)= a_0+ a_1 x$. The matrix of the finite rank operator is $2\times 2$, 
	\begin{equation*}
	M_{\mathcal{T}}= \begin{pmatrix}
	-\frac{1}{4} & - \frac{1}{6}  \\ \\
	- \frac{1}{2 \(1 +  \cot\theta\)} & 
	- \frac{3 -  \cot \theta}{12 \(1 + \cot \theta\) }
	\end{pmatrix}.
	\end{equation*}
	with eigenvalues $\lambda_1= - \frac{3+ \cot \theta  + 2 \sqrt{3+ \cot \theta+ \cot^2 \theta}}{12 \(1+ \cot \theta \) }$ and $\lambda_2=- \frac{3+ \cot \theta  - 2 \sqrt{3+ \cot \theta+ \cot^2 \theta}}{12 \(1+ \cot \theta \) }$,
	and corresponding eigenvectors
	\begin{equation*}
	u_1(x)= \frac{\cot \theta + \sqrt{3+3 \cot \theta + \cot^2 \theta}}{3}+ x
	\end{equation*}
	and
	\begin{equation*}
	u_2(x)= \frac{\cot \theta - \sqrt{3+3 \cot \theta + \cot^2 \theta}}{3}+ x. 
	\end{equation*}
\end{example}

\begin{remark}
    When $\delta-\alpha+ \alpha \delta = 0$, and by the virtue of \cref{rk:III.6}, the perturbation kernel takes the form
    \begin{equation} \label{eq:GRobin:T}
\begin{split}
T(x,y) &=-c_{00} + \(\frac{1}{2} - c_{10} \)x - \(c_{01}+\frac{1}{2} \) y - c_{20} x^2 - c_{11} x y - c_{02} y^2 - c_{30} x^3 \\
& 	\hspace{1em} - c_{21} x^2 y - c_{12} x y^2 - c_{03} y^3 -c_{31} x^3 y- c_{13} x y^3 \\
&\define	d_{00} + d_{10} x +d_{10} y +d_{20} x^2 +d_{11} x y +d_{20} y^2 + d_{30} x^3 + d_{21} x^2 y \\
& 	\hspace{1em} + d_{21} x y^2 +d_{30} y^3 +d_{31} x^3 y+ d_{31} x y^3 .
\end{split}
\end{equation}
where the $c_{ij}$ and $d_{ij}$ coefficients are again explicitly given in \cref{Appendix:A}, with $\beta=0$. 
The eigenfunctions have the form $f(x)=a_0+ a_1 x+a_2 x^2 + a_3 x^3$, and the corresponding matrix $M_{\mathcal{T}}= \left(m_{ij}\right)$, is $4\times 4$ with elements given by \cref{Appendix:B} with $\beta=0$. This is a rank 2 matrix with $\lambda=0$ as a double eigenvalue; see also \cref{sec:GSARS Boundary Conditions}. 
\end{remark}

\subsection{Periodic Boundary Conditions}
\label{sec:Periodic Problem}
The Green's function in this case is given in \cref{sec:Periodic Resolvent}, for $0\le x\le y\le 1$, by 
\begin{align*} \label{eq:integral-Periodic}
G_0(x,y)= \frac{1}{12} 
\(1+ 6 x - 6 y + 6 x^2 - 12 x y + 6 y^2\). 
\end{align*}
The expression of $T(x,y)$ is given by 
\begin{align*} 
T(x,y)=
\frac{1}{12} \(-1-6 x^2 +12 x y-6 y^2 \).
\end{align*}
It is symmetric, viz. $T(x,y)=T(y,x)$. 
Since it is a quadratic polynomial in $x$, the eigenvector of the matrix problem is of the form $f(x)=a_0+a_1 x+a_2 x^2$. 
The corresponding matrix is   
\begin{equation*}
M_{\mathcal{T}}= \begin{pmatrix}
-\frac{1}{4} & - \frac{1}{6} &  - \frac{23}{180}  \\ \\
\frac{1}{2} & \frac{1}{3} &  \frac{1}{4} \\ \\
-\frac{1}{2} & -\frac{1}{4} &   -\frac{1}{6} \\ \\ 
\end{pmatrix}.
\end{equation*}
The eigenvalues are: $\lambda_1=\frac{1}{12}$, $\lambda_2=\frac{1}{60} \left(-5 - \sqrt{30}\right)$, $\lambda_3=\frac{1}{60} \left(-5 + \sqrt{30}\right)$.
The corresponding eigenvectors are
\begin{equation}\label{eq:pPevs}
u_1(x)= -\frac{1}{2} + x, \quad
u_2(x)= \frac{1}{30} \left(10+\sqrt{30}\right)- x+ x^2, \quad
u_3(x)= \frac{1}{30} \left(10-\sqrt{30}\right)-x+x^2. 
\end{equation}
We have:
\begin{theorem}\label{thm:Periodic}
	The periodic operator $\mathcal{K_{G}}$ is a rank 3 perturbation of $\mathcal{K}$. Moreover, \begin{equation*}
	\mathcal{T} f (x) = \sum_{k=1}^3 \langle f, u_k\rangle u_k(x), 
	\end{equation*}
	where $u_k(x)$, $k=1,2,3$, are defined in \cref{eq:pPevs}.
\end{theorem}

\subsection{Anti-Periodic Boundary Conditions}
\label{sec:Anti-Periodic Problem}
The anti-periodic integral operator turns out to be the ``closest'' to the nonlocal integral operator. To see this, note that for $0\le x\le y\le 1$, the derived Green's function takes the form 
\begin{align} \label{eq:integral-anti-Periodic}
G_0(x,y)=  \frac{1}{4} \(1+ 2 x - 2y\) 
\end{align}
and $T(x,y)=-\frac{1}{4}$. Thus $\mathcal{T}$ has eigenvalue $\lambda=-\frac{1}{4}$ and eigenvector $u_1(x)\equiv 1$. 
Hence, we have:
\begin{theorem}\label{thm:Anti-Periodic}
	The anti-periodic operator $\mathcal{K_{G}}$ is a rank 1 perturbation of $\mathcal{K}$.  Moreover, 
	\begin{equation*}
	\mathcal{T} f (x) = \langle f, u_1\rangle u_1(x) = \int_0^1 f(y) \dd{y},
	\end{equation*}
	which is simply a constant often called the ``DC'' component of $f(x)$.
\end{theorem}

\subsection{Radoux Boundary Conditions}
\label{sec:Radoux Problem}
The Green's function $G_0(x,y)$ corresponding to this problem is explicitly given in \cref{sec:Radoux Resolvent} as
\begin{align*} 
G_0(x,y) = \frac{1}{10} \(
10 x - 18 xy + 5 x y^3+ 5 x^3 y\), 
\end{align*}
for $0\le x\le y \le 1$. 
That of $T(x,y)$ is given by 
\begin{align*} 
T(x,y) =
\frac{1}{10} \(
-5 x -5 y + 18  xy - 5 x y^3 - 5 x^3 y\). 
\end{align*}
This is again symmetric, viz. $T(x,y)=T(y,x)$, and we only need a single integral to calculate the corresponding matrix operator. 
\begin{hide}
	Since $T(x,y)$ is a cubic polynomial in $x$, the eigenvectors of the matrix problem are of the form $f(x)=a_0+a_1 x+a_2 x^2 +a_3 x^3$. The matrix of the finite rank operator is
	\begin{equation}
	M_{\mathcal{T}}= \begin{pmatrix}
	-\frac{1}{4} & - \frac{1}{6} &  - \frac{1}{8} &  - \frac{1}{10}  \\ \\
	\frac{11}{40} & \frac{1}{4} &  \frac{1}{5} &  \frac{229}{1400} \\ \\
	0 & 0 &   0 &  0 \\ \\ 
	-\frac{1}{4} & - \frac{1}{6} &  - \frac{1}{8} &  - \frac{1}{10} \\ \\ 
	\end{pmatrix}.
	\end{equation}
	Its eigenvalues are: $\lambda_1=\lambda_2=0$, $\lambda_3=\frac{1}{420} \left(-21- \sqrt{2982}\right)$, and $\lambda_4=\frac{1}{420} \left(-21+ \sqrt{2982}\right)$. The corresponding eigenvectors are
	\begin{equation}\label{eq:pLadoux}
	\begin{split}
	u_1(x)&=\frac{1}{8} - \frac{15}{16} x+ x^2, \\
	u_2(x)&=\frac{19}{140}- \frac{45}{56} x+ x^3, \\
	u_3(x)&=1- \frac{456 +11\sqrt{2982}}{10 \(49+ \sqrt{2982}\)}  x + x^3, \\
	u_4(x)&=1- \frac{456 -11\sqrt{2982}}{10 \(49- \sqrt{2982}\)}  x + x^3.
	\end{split}
	\end{equation}
	In conclusion, we have the following theorem. 
	\begin{theorem}\label{thm:Radoux}
		The Radoux operator $\mathcal{K_{G}}$ is a rank 4 perturbation of $\mathcal{K}$. Moreover, \begin{equation*}
		\mathcal{T} f (x) = \sum_{k=1}^4 \langle f, u_k\rangle u_k(x),
		\end{equation*}
		where $u_k(x)$, $k=1, 2, 3, 4$, are defined in \cref{eq:pLadoux}.
	\end{theorem}
\end{hide}
\begin{align*}
\mathcal{T} f(x) &= \int_0^1 \frac{1}{10} \(
-5 x -5 y + 18  xy - 5 x y^3 - 5 x^3 y\) f(y) \dd{y} \\
&= \frac{x}{10} \int_0^1 \( -5 + 18 y -5y^3\) f(y) \dd{y} 
-\frac{1+x^3}{2} \int_0^1 y f(y) \dd{y}.
\end{align*}
Hence, the eigenvalue problem corresponding to this operator, $\mathcal{T} f(x) = \lambda f(x)$, leads to a cubic form of the eigenvector $f(x)= a_0 x + a_1 \left(1 + x^3\right)$, and the matrix eigenvalue problem 
\[
\begin{pmatrix}
\frac{307}{700} & \frac{1}{4} \\ \\
-\frac{7}{20} & - \frac{1}{6}
\end{pmatrix}
\begin{pmatrix}
a_0 \\
a_1 
\end{pmatrix} \,
= \lambda
\begin{pmatrix}
a_0 \\
a_1 
\end{pmatrix} .
\]
Its eigenvalues are $\lambda_1=\frac{1}{420} \left(-21- \sqrt{2982}\right)$, and $\lambda_2=\frac{1}{420} \left(-21+ \sqrt{2982}\right)$. The corresponding eigenvectors are
\begin{equation}\label{eq:pLadoux}
\begin{split}
u_1(x) &= x - \frac{5}{921}\(126 + \sqrt{2982}\) \(1 + x^3\), \\
u_2(x) &= x + \frac{5}{921}\(-126 + \sqrt{2982}\) \(1 + x^3\).
\end{split}
\end{equation}
In conclusion, we have the following theorem. 
\begin{theorem}\label{thm:Radoux}
	The Radoux operator $\mathcal{K_{G}}$ is a rank 2 perturbation of $\mathcal{K}$. Moreover, \begin{equation*}
	\mathcal{T} f (x) = \langle f, u_1\rangle u_1(x) + \langle f, u_2\rangle  u_2(x),
	\end{equation*}
	where $u_1(x)$ and $u_2(x)$ are defined in \cref{eq:pLadoux}.
\end{theorem}

\subsection{GSARC Boundary Conditions with Non-Zero Discriminant}
\label{sec:Perturb GSARC with Non-Zero Discriminant}
For $\Delta=\delta-\alpha- 2\beta+ \beta^2 + \alpha \delta \neq 0$ the Green's function developed in \cref{sec:GSARC with Non-Zero Discriminant Resolvent} is a quadratic polynomial given, for $0\le x\le y\le 1$ by
\begin{align*}
G_0(x,y) = \frac{-1+\delta - \(\alpha + \beta\) x + \(\beta-\delta\) y+ \(\beta^2+ \alpha \delta \) x (1-y)}{\delta-\alpha- 2 \beta + \beta^2 + \alpha \delta} .
\end{align*}
The kernel of the perturbation is then
\begin{align*}
T(x,y) = c_{00} + c_{10} x+ c_{01} y + c_{11} x y  
\end{align*}
where
\begin{equation*}
c_{00}= \frac{1-\delta}{\Delta},
\end{equation*}
\begin{equation*}
c_{10}=c_{01}= \frac{\alpha+ \delta-\beta^2 - \alpha \delta}{2 \Delta},
\end{equation*}
and
\begin{equation*}
c_{11}=\frac{\beta^2 + \alpha \delta}{2 \Delta}.
\end{equation*}
The eigenvectors corresponding to this finite rank operator are of the form $f(x)=a_0+ a_1 x$, and the $2\times 2$ matrix of the operator $M_{\mathcal{T}}$ has the entry
$m_{11}=\frac{-4 - \alpha + 3 \delta + \beta^2 + \alpha \delta}{4 \Delta}$,
$m_{12}=\frac{-3 - \alpha + 2 \delta + \beta^2 + \alpha \delta}{6 \Delta}$,
$m_{21}=\frac{\alpha + \delta}{2 \Delta}$, and
$m_{22}=\frac{3 \alpha + 3 \delta + \beta^2 + \alpha \delta}{12 \Delta}$.
Its eigenvalues are
$\lambda_1=\frac{-6 - 3 \alpha + 3 \delta + \beta^2 + \alpha \delta - 2 \sqrt{\tilde{\Delta}}}{12 \Delta}$, 
$\lambda_2=\frac{-6 - 3 \alpha + 3 \delta + \beta^2 + \alpha \delta + 2 \sqrt{\tilde{\Delta}}}{12 \Delta}$, where
$$\tilde{\Delta} \define 9+ 9 \alpha + 3 \alpha^2 - 6 \beta^2 - 3 \alpha \beta^2 + \beta^4 - 9 \delta - 9 \alpha \delta + 3 \alpha^2 \delta + 3 \beta^2 \delta + 2 \alpha \beta^2 \delta + 3 \delta^2 + 3 \alpha \delta^2 + \alpha^2 \delta^2$$ 
and the corresponding eigenvectors are
$$u_1(x) = \frac{2 \(3 - \beta^2 - 3 \delta- \alpha \delta -\sqrt{\tilde{\Delta}} \)}{3 (\alpha+\delta)}+ x$$
and 
$$u_2(x) = \frac{2\(3 - \beta^2 - 3 \delta- \alpha \delta +\sqrt{\tilde{\Delta}} \)}{3 (\alpha+\delta)}+ x.$$

\subsection{GSARC Boundary Conditions with Zero Discriminant}
\label{sec:Perturb GSARC with Zero Discriminant}
The Green's function $G_0(x,y)$ is detailed in \cref{sec:GSARC with Zero Discriminant  Resolvent}. It assumes the form $G_0(x,y)=-\frac{1}{2} |x-y|- T(x,y)$, where $T(x,y)$ is the kernel of the perturbation given by the expression 
\begin{equation} \label{eq:GSARC:T}
\begin{split}
T(x,y) &=-c_{00} + \(\frac{1}{2} - c_{10} \)x - \(c_{01}+\frac{1}{2} \) y - c_{20} x^2 - c_{11} x y - c_{02} y^2 - c_{30} x^3 \\
& 	\hspace{1em} - c_{21} x^2 y - c_{12} x y^2 - c_{03} y^3 -c_{31} x^3 y- c_{13} x y^3 \\
&\define	d_{00} + d_{10} x +d_{10} y +d_{20} x^2 +d_{11} x y +d_{20} y^2 + d_{30} x^3 + d_{21} x^2 y \\
& 	\hspace{1em} + d_{21} x y^2 +d_{30} y^3 +d_{31} x^3 y+ d_{31} x y^3 .
\end{split}
\end{equation}
By virtue of \cref{rk3.4}, $T(x,y)=T(y,x)$ for all $x, y \in [0,1]$. This is also a cubic polynomial in $x$. The eigenvector of the matrix problem is of the form $f(x)=a_0+a_1 x+a_2 x^2 +a_3 x^3$. The corresponding matrix is $4 \times 4$ 
\begin{equation*}
M_{\mathcal{T}}= \begin{pmatrix}
m_{ij},
\end{pmatrix}.
\end{equation*}
where explicit forms of these entries, obtained using Wolfram Mathematica\textsuperscript{\textregistered} are given in \cref{Appendix:B}.

Its four eigenvalues $\lambda_k$, and four eigenvectors, $u_k(x)$, $k=1, 2, 3, 4$ are too complicated to list. 

We now have:
\begin{theorem}\label{thm:pGSARZero-evs}
	The GSARC operator $\mathcal{K_{G}}$ is an up-to-rank 4 perturbation of $\mathcal{K}$. Moreover, \begin{equation*}
	\mathcal{T} f (x) = \sum_{k=1}^{\rank\(\mathcal{K_{G}}\)} \langle f, u_k\rangle u_k(x).
	\end{equation*}
\end{theorem}

\begin{remark}\label{remark4.12}
	The case $\beta=-\alpha$ merits special attention since $T(x,y)$ in \cref{eq:GSARC:T} reduces to a quadratic polynomial; see also the expressions of $c_{ij}$'s in \cref{Appendix:A}. Since $\Delta=0$, we also have $\delta=-\alpha$. For $\alpha \neq 0, -1$, the Green's function reduces to 
	\begin{equation}
	G_0(x,y)= \frac{4+\alpha}{12 \(1+\alpha\)} + \frac{\alpha}{2 \(1+\alpha\)} x - \frac{2+\alpha}{2 \(1+\alpha\)} y - \frac{\alpha}{1+\alpha}     
	+\frac{1}{2} x^2 + \frac{1}{2} y^2.
	\end{equation}
	The kernel of the Riesz projection is given by $p(x,y)=-1$, and therefore 
	\begin{equation}
	T(x,y)= - \frac{4+\alpha}{12 \(1+\alpha\)} + \frac{1}{2 \(1+\alpha\)} x + \frac{1}{2 \(1+\alpha\)} y + \frac{\alpha}{1+\alpha} x y+      
	+\frac{1}{2} x^2 + \frac{1}{2} y^2.
	\end{equation}
	Proceeding as before, we obtain a rank 3 perturbation with eigenvalues 
	$\lambda_1=\frac{\alpha}{12 (1+\alpha)}$, 
	\\ 
	$\lambda_2=\frac{-5 - \sqrt{30}}{60}$, $\lambda_3=\frac{-5 + \sqrt{30}}{60}$ and corresponding eigenvectors 
	\begin{align*}
	u_1(x) &=  -\frac{1}{2} + x \\
	u_2(x) &= \frac{10 - \sqrt{30}}{30} -x + x^2\\
	u_3(x) &= \frac{10 + \sqrt{30}}{30} -x + x^2. 
	\end{align*}
	The cases $\alpha =-1$ and $\alpha=0$ have already been treated, i.e., the Kre\u{\i}n-von Neumann case (rank 4) and the Neumann case (rank 2), respectively. In light of these discussions, we note that rank 1 perturbation in the GSARC problem with $\Delta=0$ cannot occur. 
\end{remark} 

\subsection{GSARS Boundary Conditions} \label{sec:GSARS Boundary Conditions}

We briefly discuss the GSARS BC case, in the form of \cref{eq:GSARS}. It corresponds in our notation to $\beta=-\gamma=0$ (see Remark~\ref{rk I.2}). It is indeed equivalent to the generalized Robin problem treated in \cref{sec:Robin Resolvent} and \cref{sec:Perturb Robin Problem}. We focus on the dependence on the angles $\theta_0$ and $\theta_1$. The discriminant condition for not having a zero eigenvalue is $\Delta = \delta-\alpha+ \alpha \delta \neq 0$, which is equivalent to $\tilde{\Delta} = \cos (\theta_0 - \theta_1) + \cos (\theta_0+\theta_1) - 2 \sin (\theta_0 - \theta_1)\neq 0$.
In this case, the Riez projection $P\equiv 0$. For $0\le x \le y\le 1$, with 
\begin{align*}
c &=(-1+z) \sin(\sqrt{z} - \theta_0-\theta_1) + (-1+2\sqrt{z}-z) \sin(\sqrt{z} + \theta_0-\theta_1) \\
& \hspace{1em} + (-1-2 \sqrt{z}-z) \sin (\sqrt{z} - \theta_0+\theta_1) +(-1+z) \sin(\sqrt{z} + \theta_0+\theta_1) 
\end{align*}
the resolvent kernel is such that 
		\begin{align*} 
		c \, \sqrt{z} \, G(z,x,y)
		\end{align*}
		is equal to 
		\begin{align*} 
  \left((-1 - \sqrt{z}) \sin(\sqrt{z} \, x - \theta_0) + (-1+ \sqrt{z}) \sin(\sqrt{z} \,  x + \theta_0) \right) 
		\, \left((-1 - \sqrt{z}) \sin(\sqrt{z} \, (-1+y) - \theta_1) + (-1+ \sqrt{z}) \sin(\sqrt{z} \, (-1+y) + \theta_1)\right).
		\end{align*}
The Green's function corresponding to the GSARS BC with non-zero discriminant is then
\begin{equation*}
G_0(x,y)= \frac{2 \left(x \cos\theta_0 - \sin \theta_0 \right) \, \left((1-y) \, \cos \theta_1  + \sin \theta_1\right) }{\cos (\theta_0 - \theta_1) + \cos (\theta_0+\theta_1) - 2 \sin (\theta_0 - \theta_1)}. 
\end{equation*}
The kernel of the perturbation is given by
\begin{equation*}
T(x,y)= \frac{\left(\cos\theta_0 \cos \theta_1 + \sin (\theta_0+\theta_1)\right) (x+y) - 2 x y \cos \theta_0 \cos \theta_1 }{\cos (\theta_0 - \theta_1) + \cos (\theta_0+\theta_1) - 2 \sin (\theta_0 - \theta_1)}. 
\end{equation*}
The matrix of the associated finite rank operator is $2\times 2$, and is given by 
	\begin{equation*}
	M_{\mathcal{T}}= \begin{pmatrix}
	\frac{3 \cos (\theta_0 - \theta_1) - 5 \cos (\theta_0 + \theta_1)+ 4 \sin (\theta_0 - \theta_1)+2 \sin (\theta_0 + \theta_1)}{4 \left(\cos (\theta_0 - \theta_1) + \cos (\theta_0+\theta_1) - 2 \sin (\theta_0 - \theta_1)\right)} & \frac{2\cos (\theta_0-\theta_1)- 4 \cos (\theta_0+\theta_1)+ 3 \sin (\theta_0-\theta_1)+ \sin (\theta_0+\theta_1)}{6 \left(\cos (\theta_0 - \theta_1) + \cos (\theta_0+\theta_1) - 2 \sin (\theta_0 - \theta_1)\right)}  \\ \\
	 \frac{- \sin(\theta_0+\theta_1)}{\cos (\theta_0 - \theta_1) + \cos (\theta_0+\theta_1) - 2 \sin (\theta_0 - \theta_1)} & 
	- \frac{\cos (\theta_0 - \theta_1)+ \cos (\theta_0 + \theta_1)- 6 \sin (\theta_0 + \theta_1)}{12 \left(\cos (\theta_0 - \theta_1) + \cos (\theta_0+\theta_1) - 2 \sin (\theta_0 - \theta_1)\right) }
	\end{pmatrix}.
	\end{equation*}
The eigenvalues of the rank 2 perturbation in this case, are solutions of the quadratic equation
\begin{equation*}
    A \lambda^2 + 2 B \lambda+ C=0
\end{equation*}
where $A=48 \, \tilde{\Delta}$, $B=-20 \cos (\theta_0 - \theta_1) + 29 \cos (\theta_0 + \theta_1)- 24 \sin (\theta_0 - \theta_1)$, and $C=-\tilde{\Delta}$. 

The various GSARS cases are displayed in ~\cref{fig:GSARS} for $(\theta_0, \theta_1) \in [0, \pi)\times [0, \pi)$: Dirichlet $(0,0)$, Neumann $(\pi/2, \pi/2)$, Robin $(0, \pi/2)$,  and Radoux $(0, 3\pi/4)$. The curve represents the level sets of the discriminant condition $\tilde{\Delta}=0$. 

The case $(\theta_0, \theta_1)=(\pi/4,0)$ represents an interesting case not treated in the literature, with $\tilde{\Delta}=0$, corresponding to the boundary conditions $u'(0)+ u(0)=u(1)=0$. Its resolvent has the kernel 
\begin{equation*}
    G(z,x,y)= \frac{\left(\sqrt{z} \cos (\sqrt{z} \, x) - \sin (\sqrt{z} \, x) \right) \, \left(\sin(\sqrt{z} (1-y) ) \right)}{z \cos \sqrt{z}- \sqrt{z} \sin \sqrt{z}}.
\end{equation*}
Its Riesz projection has the kernel $p(x,y) =-3 (1-x) (1-y)$, and the nilpotent operator $D\equiv 0$.  The Green's function is given by 
\begin{equation*}
    G_0(x,y)= \frac{1}{10} \, \left(2- 2x + 15x^2 - 5x^3- 12y +12x y -15 x^2 y + 5 x^3 y + 15 y^2 - 15 x y^2 - 5 y^3+ 5 x y^3\right). 
\end{equation*}
Working as before, we obtain the kernel of the perturbation
\begin{equation} \label{T:pi}
    T(x,y)= \frac{1}{10} \, \left(-2 + 7 x- 15 x^2 + 5 x^3 + 7 y - 12 x y +15 x^2 y- 5 x^3 y - 15 y^2 + 15 x y^2 + 5 y^3 - 5 x y^3\right). 
\end{equation}
Note that $T(x,y)=T(y,x)$. The operator $\mathcal{T}$ then acts as
\begin{align*}\label{eq:GSARS-Special-perturb}
\mathcal{T} f (x) &= 
\int_{0}^1 T(x,y) f(y) \dd{y} \notag \\
&= \left(\int_{0}^1 \frac{1}{10} \left(-2+ 7 y - 15 y^2 + 5 y^3\right) f(y) \dd{y} \right) \left(1-x\right) \\
& \hspace{1em} + \left(\int_0^1 \frac{1}{10} \left(1-y\right) f(y) \dd{y}\right) \left(5 x- 15 x^2 + 5 x^3\right).
\end{align*}
The eigenvalue problem $\mathcal{T}f(x) = \lambda f(x)$
leads to a linear eigenfunctions of the form $f(x) = a_0 (-1-x) +a_1 (5 x- 15 x^2 +5 x^3)$ and the matrix eigenvalue problem
\[\begin{pmatrix}
-\frac{1}{12} & \frac{199}{420} \\
\frac{1}{30} & -\frac{1}{70}
\end{pmatrix} \,  \begin{pmatrix}
a_0  \\
a_1 \end{pmatrix} = \lambda \begin{pmatrix}
a_0  \\
a_1 
\end{pmatrix} . \]
Thus $\lambda_1=\frac{-21-\sqrt{2982}}{420}$, $\lambda_2=\frac{-21+\sqrt{2982}}{420}$, and the corresponding eigenvectors are
\begin{equation}\label{eq:pi}
u_1(x)=\frac{-14-\sqrt{2982}}{14} (1-x) + (5 x- 15 x^2 +5 x^3), \quad u_2(x)=\frac{-14+\sqrt{2982}}{14} (1-x) + (5 x- 15 x^2 +5 x^3).
\end{equation}

\begin{remark}
    The expression of the perturbation in \cref{T:pi} is cubic in $x$. An alternative analysis with an eigenfunction of the form $f(x)=a_0+a_1 x+a_2 x^2 +a_3 x^3$, with a corresponding matrix
    \begin{equation}
M_{\mathcal{T}}= \begin{pmatrix}
-\frac{9}{40} & - \frac{17}{120} &  - \frac{13}{120} &  - \frac{31}{350}  \\ \\
\frac{19}{40} & \frac{9}{40} &  \frac{3}{20} &  \frac{159}{1400} \\ \\
-\frac{3}{4} & -\frac{1}{4} &   -\frac{1}{8} &  -\frac{3}{40} \\ \\ 
\frac{1}{4} &  \frac{1}{12} & \frac{1}{24} &  \frac{1}{40} \\ \\ 
\end{pmatrix}.
\end{equation}
This leads to the same eigenvalues and eigefunctions as above, and $\lambda_0=0$, a double eigenvalue. Its SVD gives two non-zero singular values, hence a rank of 2. Notice that the determinant of $M_{\mathcal{T}}$ is zero, and so is the case of the $3\times 3$ subdeterminant extracted from the first three columns (but not the $2\times2$ subdeterminant of the first two columns), another way to conclude that the rank is indeed 2. This is a general fact for every point on the level sets in Fig.~\ref{fig:GSARS}.
\end{remark}

\begin{remark}
The discriminant condition $\tilde{\Delta}=0$ corresponds to the level sets on Fig.~\ref{fig:GSARS}. It is equivalent to $\tan \theta_1=-1+ \tan \theta_0$, leading to two possibilities: $\theta_1=\pi+\arctan\left(-1+\tan \theta_0\right)$ for $\theta_0 \in [0, \pi/4)\cup (\pi/2, \pi)$, and $\theta_1=\arctan\left(-1+\tan \theta_0\right)$ for $\theta_0 \in [\pi/4, \pi/2)$. The case $(\theta_0, \theta_1)=(\pi/2,\pi/2)$ corresponds to Neumann BCs. Curious enough, along these level sets the expressions of resolvent, Riesz projection, Green's function, and perturbation kernel are different from the calculations in this subsection (since the discriminant is zero). The problem leads to a $4\times 4$ matrix $M_{\mathcal{T}}$, but the rank of the perturbation is 2. The calculations are no different from the case $(\theta_0, \theta_1)=(\pi/4,0)$. The details of these expressions are in \cref{Appendix:C}.
\end{remark}

\begin{center}
\begin{figure}[tbhp!]
    \centering
    \includegraphics[width=0.6 \textwidth]{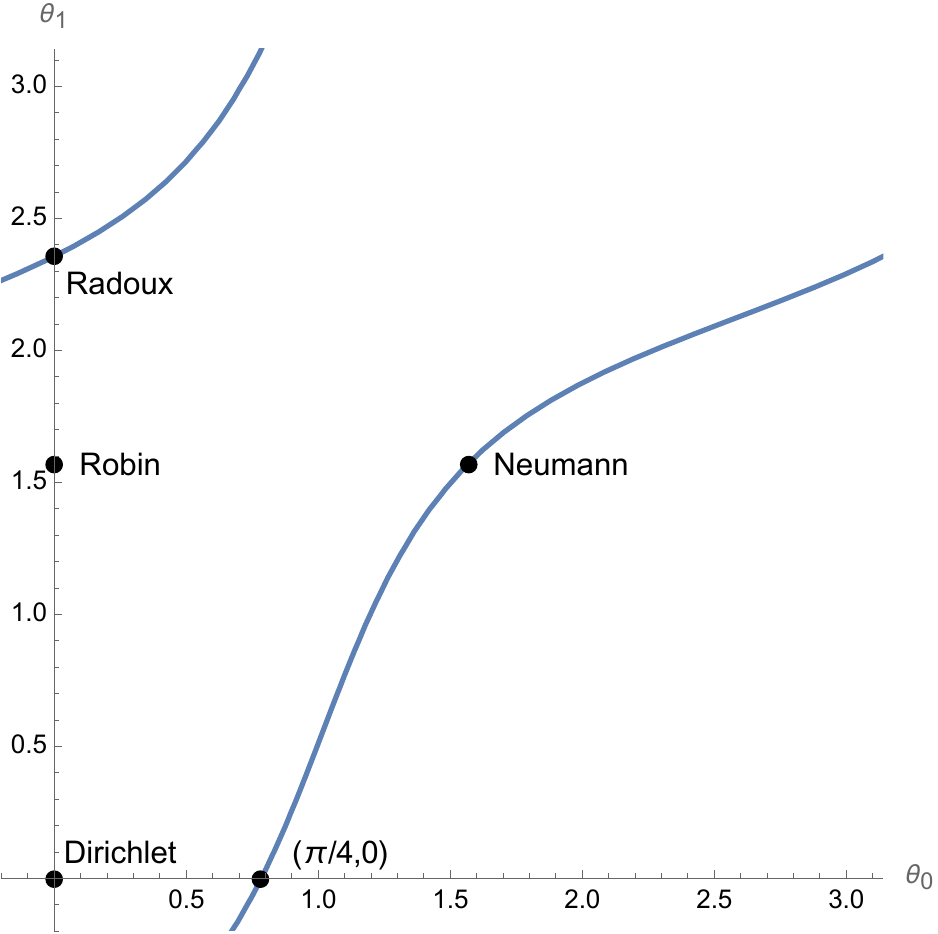}
    \caption{Examples of GSARS cases treated with respect to the discriminant condition $\cos (\theta_0 - \theta_1) + \cos (\theta_0+\theta_1) - 2 \sin (\theta_0 - \theta_1)=0$, for $\theta_0, \theta_1 \in [0, \pi)$.}
    \label{fig:GSARS}
\end{figure}
\end{center}

\subsection{Square of the Volterra Operator}
\label{sec:Perturb Square-Volt}
The earliest treatment of the eigenvalues of the operator \cref{eq:Saito-Operator} appears in the seminal paper of Lidski\u{\i} \cite{Lidskii59}, in the context of the Volterra operator--specifically the context of nonselfadjoint operators not satisfying the completeness property for Hibert-Schmidt operators. We find it also in the book of Gohberg and Kre\u{\i}n \cite[pp.~208--210]{GohbergKrein1969}, where it is referred to as a ``Volterra operator with a two dimensional imaginary component''. 

With $\mathcal{V} f(x):=\int_x^1 f(y) \dd{y}$,  one immediately obtains $\mathcal{V}^2 f(x)=\int_x^1 \left(x- y\right)f(y) \dd{y}$. This operator is then decomposed into its Hermitian components \cite{Lidskii59,GohbergKrein1969,ArlinTsek2006}: 
\begin{equation}\label{eq:Volt1}
\mathcal{V}^2 = \left(\mathcal{V}^2\right)_{R}+ \im \, \left(\mathcal{V}^2\right)_{I}
\end{equation}
where 
\begin{equation} \label{eq:Volt2}
\begin{split}
\left(\mathcal{V}^2\right)_{R} \, f &:=\frac{1}{2} \left(\mathcal{V}^2 + \mathcal{V}^{2 \ast}\right) \, f= - \frac{1}{2} \, \int_0^1 |x-y| f(y) \dd{y} , \\
\left(\mathcal{V}^2\right)_{I} \, f &:=\frac{1}{2i} \left(\mathcal{V}^2 - \mathcal{V}^{2 \ast} \right) \, f= \frac{1}{2\im} \, \int_0^1 \left(x-y\right) f(y) \dd{y} .
\end{split}
\end{equation}
The real part $\left(\mathcal{V}^2\right)_{R}$ is clearly our operator $\mathcal{K}$ in \cref{eq:Saito-Operator}. In this light, $\mathcal{K}= \mathcal{V}^2+\mathcal{T_V}$ where $\mathcal{T_V}\, f=-\im 	\left(\mathcal{V}^2\right)_{I} \, f  = - \frac{1}{2} \, \int_0^1 \left(x-y\right) f(y) \dd{y}.$

\begin{theorem}\label{thm:volterra}
	The operator $\mathcal{T_V}$ is a rank 2 perturbation of $\mathcal{K}$. 
\end{theorem}
\begin{proof}
	To see this, we calculate the eigenvalues and eigenvectors of this perturbation, i.e., solve the eigenvalue problem $\mathcal{T_V}\, f = \lambda f$. This statement is equivalent to the problem
	\begin{equation*}
	- \frac{1}{2} \, \int_0^1 \left(x-y\right) f(y) \dd{y} = \lambda \, f(x). 
	\end{equation*}
	Note that $\lambda\neq 0$ otherwise $f(x)\equiv 0$. This immediately implies that $f(x) = m \, x +b$, where $m= - \frac{1}{2\lambda} \, \int_0^1 \, f(y) \dd{y}$ and $b=-\frac{1}{2 \lambda} \, \int_0^1 \, y \, f(y) \dd{y}$. The problem then reduces to the matrix eigenvalue problem
	\[\begin{pmatrix}
	-\frac{1}{4} & - \frac{1}{2} \\
	\frac{1}{6} & \frac{1}{4}
	\end{pmatrix} \,  \begin{pmatrix}
	m  \\
	b \end{pmatrix} = \lambda \begin{pmatrix}
	m  \\
	b 
	\end{pmatrix} .\]
	Thus $\lambda_1=\frac{\im}{2\sqrt{12}}$, $\lambda_2=-\frac{\im}{2\sqrt{12}}$, and the corresponding complex-valued eigenvectors 
	$u_1(x)=\frac{1}{2}\left(-3+ \im\sqrt{3}\right) x+ 1$, 
	$u_2(x)=\frac{1}{2}\left(-3- \im\sqrt{3}\right) x+ 1$. This agrees with Lidski\u{\i} \cite{Lidskii59}.  
\end{proof}

\section{Conclusion}

In this article we offered a unified framework of looking at general real coupled self-adjoint BVPs, where the use of the resolvent kernel was essential to the analysis, and where the Green's formula corresponding to the integral operator formulation of Sturm-Liouville problems was obtained using an abstract and very classical formulation in \cite{KATO}. While apparently similar in formulation to the nonlocal integral operator commuting with the Laplacian corresponding to the free space Green's function, the Kre\u{\i}n-von Neumann problem was shown to be the ``farthest'', while the anti-periodic problem proved to be the ``closest'' when seen as finite-rank perturbations. In \cite{HSIteratedKernel2}, we will develop the spectral theory of the associate 
iterated \emph{Brownian bridge kernels} corresponding to this GSARC framework, and 
show how to recover the values of the spectral zeta function for the nonlocal operator \cref{eq:Saito-Operator} from the power series of the resolvent in a unified way; see also 
\cite{CavFassMcC15,FasshauerBook,AGHKLT19,AGHKLT18}.



\appendix
\section{Details of the Green's Function for GSARC BCs with \texorpdfstring{$\Delta$=0}{}}
\label{Appendix:A}
The coefficients of $G_0(x,y)$ were generated using Wolfram Mathematica\textsuperscript{\textregistered}. Noting also \cref{rk3.4}, they are given by the following equations. 
\begin{hide}
	\begin{align*}
	20 (3+\alpha^2 +\beta^2 - \alpha \beta + 3 \alpha - 3 \beta)^2 \, c_{00}
	&= 60+ 120 \alpha + 84 \alpha^2 + 20 \alpha^3-75 \beta \\
	&  
	- 87 \alpha \beta - 28 \alpha^2 \beta +54 \beta^2  \\
	&  
	+ 34 \alpha \beta^2 -23 \beta^3 - 7 \alpha \beta^3 +4 \beta^4
	\end{align*}
\end{hide}

\begin{align*}
20 (3+\alpha^2 +\beta^2 - \alpha \beta + 3 \alpha - 3 \beta)^2 \, c_{00}
&= 4 \(15 + 30 \alpha + 21 \alpha^2 +5 \alpha^3\) \\
&  
- \(75 + 87 \alpha + 28 \alpha^2\) \beta   \\
&  
+ 2 \(27+17 \alpha+ 2 \alpha^2\) \beta^2- \(23 + 7 \alpha\) \beta^3 +4 \beta^4
\end{align*}

\begin{hide}
	\begin{align*}
	20 (3+\alpha^2 +\beta^2 - \alpha \beta + 3 \alpha - 3 \beta)^2 \, c_{10}
	&= 60 \alpha + 120 \alpha^2 + 84 \alpha^3 + 20 \alpha^4 - 30 \beta \\
	& - 135 \alpha \beta - 123 \alpha^2 \beta-34 \alpha^3 \beta +15 \beta^2 \\
	& + 72 \alpha \beta^2 + 33 \alpha^2 \beta^2 + 9 \beta^3 - 7 \alpha \beta^3 - 4 \beta^4
	\end{align*}
\end{hide}

\begin{align*}
20 (3+\alpha^2 +\beta^2 - \alpha \beta + 3 \alpha - 3 \beta)^2 \, c_{10}
&=  4 \alpha \(15+ 30 \alpha + 21\alpha^2 + 5 \alpha^3\) \\
& -\(30 + 135 \alpha + 123 \alpha^2+ 34\alpha^3\) \beta \\
& + 3 \( 5+ 24 \alpha+ 11 \alpha^2\) \beta^2 + \(9 - 7 \alpha\) \beta^3 - 4 \beta^4
\end{align*}

\begin{hide}
	\begin{align*}
	\frac{20}{3} (3+\alpha^2 +\beta^2 - \alpha \beta + 3 \alpha - 3 \beta)^2 \, c_{01} &= -60 - 100 \alpha -60 \alpha^2-12 \alpha^3 +110 \beta \\
	& +115 \alpha \beta + 39 \alpha^2 \beta+2 \alpha^3 \beta -95 \beta^2 \\
	& -56 \alpha \beta^2 -9 \alpha^2 \beta^2 + 43 \beta^3 +11 \alpha \beta^3 - 8 \beta^4
	\end{align*}
\end{hide}

\begin{align*}
\frac{20}{3} (3+\alpha^2 +\beta^2 - \alpha \beta + 3 \alpha - 3 \beta)^2 \, c_{01} &= 
-4 \(15+ 25 \alpha + 15 \alpha^2 + 3 \alpha^3 \) \\
& + \(110+115 \alpha + 39 \alpha^2 + 2 \alpha^3\) \beta \\
& - \(95+ 56 \alpha + 9 \alpha^2\) \beta^2 + \(43 +11 \alpha\) \beta^3 - 8 \beta^4
\end{align*}

\begin{hide}
	\begin{align*}
	\frac{20}{3}  (3+\alpha^2 +\beta^2 - \alpha \beta + 3 \alpha - 3 \beta)^2 \, c_{11} &= - 60 \alpha - 100 \alpha^2 - 60 \alpha^3 - 12 \alpha^4 + 100 \alpha \beta \\
	&  95 \alpha^2 \beta+27 \alpha^3 \beta + 20 \beta^2 \\
	& - 50 \alpha \beta^2 - 22 \alpha^2 \beta^2 - 25 \beta^3 + 7 \alpha \beta^3 + 8 \beta^4
	\end{align*}
\end{hide}

\begin{align*}
\frac{20}{3}  (3+\alpha^2 +\beta^2 - \alpha \beta + 3 \alpha - 3 \beta)^2 \, c_{11} &= 
-4 \alpha \(15+ 25 \alpha + 15 \alpha^2 + 3 \alpha^3\)\\
&   +\alpha \(100+ 95 \alpha+27 \alpha^2\) \beta  \\
&  +2 \(10- 25 \alpha - 11 \alpha^2\) \beta^2 - \(25 -7\alpha\) \beta^3 + 8 \beta^4
\end{align*}

\begin{align*}
c_{20} = c_{02}= \frac{3  (1-\beta)^2}{2 (3+\alpha^2 +\beta^2 - \alpha \beta + 3 \alpha - 3 \beta)}
\end{align*}

\begin{align*}
c_{30}=c_{03}= \frac{\(1-\beta\) \(\alpha+\beta\) }{2 (3+\alpha^2 +\beta^2 - \alpha \beta + 3 \alpha - 3 \beta)}
\end{align*}

\begin{align*}
c_{21}=	c_{12}=  \frac{3 \(1-\beta\) \(\alpha+\beta\)}{2 (3+\alpha^2 +\beta^2 - \alpha \beta + 3 \alpha - 3 \beta)}
\end{align*}

\begin{align*}c_{31}=c_{13}	= \frac{\(\alpha+ \beta\)^2}{2 (3+\alpha^2 +\beta^2 - \alpha \beta + 3 \alpha - 3 \beta)}
\end{align*}

\section{Details of the perturbation kernel for GSARC BCs with \texorpdfstring{$\Delta$=0}{}}\label{Appendix:B}

The following is the list of entries of the matrix associated with the GSARC case with zero discriminant in \cref{sec:Perturb GSARC with Zero Discriminant}. They were calculated using Wolfram Mathematica\textsuperscript{\textregistered}. The $d_{ij}$ coefficients are defined by \cref{eq:GSARC:T}. They are explicitly given by 
\begin{align}\label{eq:d-coeff}
d_{00}&=-c_{00} \notag \\
d_{10}&=\frac{1}{2} - c_{10}= -c_{01}- \frac{1}{2} \notag \\
d_{20}&=-c_{20}=-c_{02} \notag \\
d_{11}&=-c_{11} \notag \\
d_{30}&=-c_{30}=-c_{03} \notag \\
d_{21}&=-c_{21}=-c_{12}\notag \\
d_{31}&=-c_{31}=-c_{13}
\end{align}

\begin{equation*}
m_{11}=d_{00}+ \frac{d_{10}}{2}+\frac{d_{20}}{3}+\frac{d_{30}}{4}
\end{equation*}

\begin{equation*}
m_{12}=\frac{d_{00}}{2}+ \frac{d_{10}}{3}+\frac{d_{20}}{4}+\frac{d_{30}}{5}
\end{equation*}

\begin{equation*}
m_{13}=\frac{d_{00}}{3}+ \frac{d_{10}}{4}+\frac{d_{20}}{5}+\frac{d_{30}}{6}
\end{equation*}

\begin{equation*}
m_{14}=\frac{d_{00}}{4}+ \frac{d_{10}}{5}+\frac{d_{20}}{6}+\frac{d_{30}}{7}
\end{equation*}

\begin{equation*}
m_{21}=d_{10}+ \frac{d_{11}}{2}+\frac{d_{31}}{4}
\end{equation*}

\begin{equation*}
m_{22}=\frac{d_{10}}{2}+ \frac{d_{11}}{3}+\frac{d_{31}}{5}
\end{equation*}	

\begin{equation*}
m_{23}=\frac{d_{10}}{3}+ \frac{d_{11}}{4}+\frac{d_{31}}{6}
\end{equation*}

\begin{equation*}
m_{24}=\frac{d_{10}}{4}+ \frac{d_{11}}{5}+\frac{d_{31}}{7}
\end{equation*}

\begin{equation*}
m_{31}=d_{20}
\end{equation*}	

\begin{equation*}
m_{32}=\frac{d_{20}}{2}
\end{equation*}

\begin{equation*}
m_{33}=\frac{d_{20}}{3}
\end{equation*}

\begin{equation*}
m_{34}=\frac{d_{20}}{4}
\end{equation*}

\begin{equation*}
m_{41}=d_{30}+ \frac{d_{31}}{2}
\end{equation*}

\begin{equation*}
m_{42}=\frac{d_{30}}{2}+ \frac{d_{31}}{3}
\end{equation*}

\begin{equation*}
m_{43}=\frac{d_{30}}{3}+ \frac{d_{31}}{4}
\end{equation*}

\begin{equation*}
m_{44}=\frac{d_{30}}{4}+ \frac{d_{31}}{5}.
\end{equation*}

In terms of the GSARC parameters, $\alpha,\beta$ (when the discriminant is equal to zero), after a lot of simplifications using Wolfram Mathematica\textsuperscript{\textregistered}, the matrix elements reduce to the following:

\begin{hide}
	\begin{align*}
	-40 \(3 + 3 \alpha -3 \beta - \alpha \beta + \alpha^2 + \beta^2\)^2 \, m_{11} &= 
	90 + 195 \alpha + 173 \alpha^2 + 69 \alpha^3 + 10 \alpha^4 -165 \beta \\
	&  - 224 \alpha \beta - 114 \alpha^2 \beta - 19 \alpha^3 \beta + 143 \beta^2 \\
	&  + 120 \alpha \beta^2 + 31 \alpha^2 \beta^2 - 57 \beta^3 - 21 \alpha \beta^3 + 9\beta^4
	\end{align*}
\end{hide}
\begin{align*}
40 \(3 + 3 \alpha -3 \beta - \alpha \beta + \alpha^2 + \beta^2\)^2 \, m_{11}
&= - \(90 + 195 \alpha + 173 \alpha^2 + 69 \alpha^3 + 10 \alpha^4 \)  \\
&  + \(165 + 224 \alpha + 114 \alpha^2 + 19 \alpha^3\) \beta  \\
&  - \(143 + 120 \alpha + 31 \alpha^2 \) \beta^2 +3\(19 +7 \alpha\) \beta^3 - 9\beta^4
\end{align*}

\begin{hide}
	\begin{align*}
	-120 \(3 + 3 \alpha -3 \beta - \alpha \beta + \alpha^2 + \beta^2\)^2 \, m_{12} &=
	135 + 291 \alpha + 273 \alpha^2 + 120 \alpha^3 + 20 \alpha^4 -294 \beta \\
	&  - 402 \alpha \beta - 216 \alpha^2 \beta - 40 \alpha^3 \beta + 270 \beta^2 \\
	&  + 231 \alpha \beta^2 + 63 \alpha^2 \beta^2 - 108 \beta^3 - 40 \alpha \beta^3 + 17 \beta^4
	\end{align*}    
\end{hide}
\begin{align*}
120 \(3 + 3 \alpha -3 \beta - \alpha \beta + \alpha^2 + \beta^2\)^2 \, m_{12}
&= -\(135 + 291 \alpha + 273 \alpha^2 + 120 \alpha^3 + 20 \alpha^4\)\\
&  +2 \(147 + 201 \alpha + 108 \alpha^2 +20 \alpha^3 \) \beta \\
&  -3 \(90  + 77 \alpha  + 21 \alpha^2 \) \beta^2 + 4 \(27+ 10 \alpha \) \beta^3- 17 \beta^4
\end{align*}

\begin{hide}
	\begin{align*}
	-240 \(3 + 3 \alpha -3 \beta - \alpha \beta + \alpha^2 + \beta^2\)^2 \, m_{13}
	&= 186 + 396 \alpha + 378 \alpha^2 + 172 \alpha^3 + 30 \alpha^4 -438 \beta \\
	&  - 597 \alpha \beta - 325 \alpha^2 \beta - 62 \alpha^3 \beta +411 \beta^2 \\
	&  + 352 \alpha \beta^2 + 97 \alpha^2 \beta^2 - 165 \beta^3 - 61 \alpha \beta^3 + 26 \beta^4
	\end{align*}    
\end{hide}
\begin{align*}
240 \(3 + 3 \alpha -3 \beta - \alpha \beta + \alpha^2 + \beta^2\)^2 \, m_{13}
&= -2 \(93 + 198 \alpha + 189 \alpha^2 + 86 \alpha^3 + 15 \alpha^4\)  \\
&  +\(438  + 597 \alpha + 325 \alpha^2 + 62 \alpha^3\) \beta \\
&  - \(411  + 352 \alpha  + 97 \alpha^2 \) \beta^2 + \(165 + 61 \alpha\) \beta^3 - 26 \beta^4
\end{align*}

\begin{hide}
	\begin{align*}
	- 2800 \(3 + 3 \alpha -3 \beta - \alpha \beta + \alpha^2 + \beta^2\)^2 \, m_{14} 
	&= 1680 + 3540 \alpha + 3400 \alpha^2 + 1572 \alpha^3 + 280 \alpha^4 -4125 \beta \\
	&  - 5605 \alpha \beta - 3064 \alpha^2 \beta - 592 \alpha^3 \beta +3910 \beta^2 \\
	&  + 3346 \alpha \beta^2 + 924 \alpha^2 \beta^2 - 1573 \beta^3 - 581 \alpha \beta^3 - 248 \beta^4
	\end{align*}   
\end{hide}
\begin{align*}
2800 \(3 + 3 \alpha -3 \beta - \alpha \beta + \alpha^2 + \beta^2\)^2 \, m_{14} 
&= - 4\(420 + 885 \alpha + 850 \alpha^2 + 393 \alpha^3 + 70 \alpha^4 \)  \\
&   + \(4125 + 5605 \alpha + 3064 \alpha^2 + 592 \alpha^3\) \beta  \\
&  - 2 \(1955 + 1673 \alpha + 462 \alpha^2\) \beta^2  
+ \(1573+ 581 \alpha \) \beta^3   - 248 \beta^4
\end{align*}

\begin{hide}
	\begin{align*}
	40 \(3 + 3 \alpha -3 \beta - \alpha \beta + \alpha^2 + \beta^2\)^2 \, m_{21}
	&= 180 + 420 \alpha + 345 \alpha^2 + 117 \alpha^3 + 11 \alpha^4 -300 \beta \\
	& - 540 \alpha \beta - 294 \alpha^2 \beta - 58 \alpha^3 \beta + 195 \beta^2 \\
	& + 261 \alpha \beta^2 + 60 \alpha^2 \beta^2 - 48 \beta^3 - 52 \alpha \beta^3 - \beta^4
	\end{align*}    
\end{hide}
\begin{align*}
40 \(3 + 3 \alpha -3 \beta - \alpha \beta + \alpha^2 + \beta^2\)^2 \, m_{21}
&= \(180 + 420 \alpha + 345 \alpha^2 + 117 \alpha^3 + 11 \alpha^4 \)\\
& -  2 \(150 + 270 \alpha + 147 \alpha^2+ 29 \alpha^3\) \beta \\
& + 3 \(65 + 87 \alpha + 20 \alpha^2\) \beta^2  - 4 \(12 + 13 \alpha\) \beta^3  - \beta^4
\end{align*}

\begin{hide}
	\begin{align*}
	40 \(3 + 3 \alpha -3 \beta - \alpha \beta + \alpha^2 + \beta^2\)^2 \, m_{22}
	&= 90 + 240 \alpha + 218 \alpha^2 + 84 \alpha^3 + 10 \alpha^4 -150 \beta \\
	& - 329 \alpha \beta - 199 \alpha^2 \beta - 44 \alpha^3 \beta + 83 \beta^2\\
	& +160 \alpha \beta^2 + 41 \alpha^2 \beta^2 - 7 \beta^3 - 31 \alpha \beta^3 -6 \beta^4
	\end{align*}
\end{hide}
\begin{align*}
40 \(3 + 3 \alpha -3 \beta - \alpha \beta + \alpha^2 + \beta^2\)^2 \, m_{22}
&= 2 \(45 + 120 \alpha + 109 \alpha^2 + 42 \alpha^3 + 5 \alpha^4\) \\
&  - \(150 + 329 \alpha + 199 \alpha^2 +44 \alpha^3\)  \beta \\
&  +\(83+160 \alpha + 41 \alpha^2\)\beta^2 - \(7 + 31 \alpha \) \beta^3 -6 \beta^4
\end{align*}

\begin{hide}
	\begin{align*}
	80 \(3 + 3 \alpha -3 \beta - \alpha \beta + \alpha^2 + \beta^2\)^2 \, m_{23} &= 
	120 + 340 \alpha + 320 \alpha^2 + 128 \alpha^3 + 16 \alpha^4 -200 \beta \\
	&  - 480 \alpha \beta - 301 \alpha^2 \beta - 69 \alpha^3 \beta + 100 \beta^2 \\
	&  + 234 \alpha \beta^2 + 62 \alpha^2 \beta^2 +3 \beta^3 - 45 \alpha \beta^3 - 12\beta^4
	\end{align*}  
\end{hide}
\begin{align*}
80 \(3 + 3 \alpha -3 \beta - \alpha \beta + \alpha^2 + \beta^2\)^2 \, m_{23}
&= 4 \(30 + 85 \alpha + 80 \alpha^2 + 32 \alpha^3 + 4 \alpha^4 \)\\
&  - \(200+ + 480 \alpha +301 \alpha^2+69 \alpha^3\) \beta\\
&  + 2\(50 + 117 \alpha + 31 \alpha^2\) \beta^2 + 3 \(1- 15 \alpha\) \beta^3 - 12\beta^4
\end{align*}

\begin{hide}
	\begin{align*}
	2800 \(3 + 3 \alpha -3 \beta - \alpha \beta + \alpha^2 + \beta^2\)^2 \, m_{24} &= 
	3150 + 9240 \alpha + 8850 \alpha^2 + 3600 \alpha^3 + 458 \alpha^4 \\
	&  -5250 \beta - 13275 \alpha \beta - 8475 \alpha^2 \beta - 1978 \alpha^3 \beta \\
	&  + 2445 \beta^2 + 6480 \alpha \beta^2 + 1743 \alpha^2 \beta^2 +285 \beta^3 
	\\ 
	&  - 1243 \alpha \beta^3 - 382 \beta^4
	\end{align*}    
\end{hide}

\begin{align*}
2800 \(3 + 3 \alpha -3 \beta - \alpha \beta + \alpha^2 + \beta^2\)^2 \, m_{24} &= 
2 \(1575 +4620 \alpha +4425 \alpha^2+1800 \alpha^3+229 \alpha^4\) \\
&  - \(5250 +13275 \alpha +8475 \alpha^2+ 1978 \alpha^3\) \beta \\
&  + 3 \(815 +2160\alpha +581 \alpha^2\) \beta^2 + \(285-1243 \alpha\) \beta^3 - 382 \beta^4
\end{align*}

\begin{equation*}
m_{31}=-\frac{3 \(1-\beta\)^2}{2 \(3 + 3 \alpha -3 \beta - \alpha \beta + \alpha^2 + \beta^2\)}
\end{equation*}

\begin{equation*}
m_{32}=-\frac{3 \(1-\beta\)^2}{4 \(3 + 3 \alpha -3 \beta - \alpha \beta + \alpha^2 + \beta^2\)}
\end{equation*}

\begin{equation*}
m_{33}=-\frac{\(1-\beta\)^2}{2 \(3 + 3 \alpha -3 \beta - \alpha \beta + \alpha^2 + \beta^2\)}
\end{equation*}

\begin{equation*}
m_{34}=-\frac{3 \(1-\beta\)^2}{8 \(3 + 3 \alpha -3 \beta - \alpha \beta + \alpha^2 + \beta^2\)}
\end{equation*}

\begin{equation*}
m_{41}=-\frac{\(2+ \alpha - \beta\) \(\alpha+ \beta\)}{4 \(3 + 3 \alpha -3 \beta - \alpha \beta + \alpha^2 + \beta^2\)}
\end{equation*}

\begin{equation*}
m_{42}=-\frac{\(3+ 2 \alpha - \beta\) \(\alpha+ \beta\)}{12 \(3 + 3 \alpha -3 \beta - \alpha \beta + \alpha^2 + \beta^2\)}
\end{equation*}

\begin{equation*}
m_{43}=-\frac{\(4+ 3\alpha - \beta\) \(\alpha+ \beta\)}{24 \(3 + 3 \alpha -3 \beta - \alpha \beta + \alpha^2 + \beta^2\)}
\end{equation*}

\begin{equation*}
m_{44}=-\frac{\(5+ 4\alpha - \beta\) \(\alpha+ \beta\)}{40 \(3 + 3 \alpha -3 \beta - \alpha \beta + \alpha^2 + \beta^2\)}.
\end{equation*}

\section{Details of the perturbation kernel for GSARS BCs with \texorpdfstring{$\tilde{\Delta}$=0}{}}\label{Appendix:C}

In either case $\theta_1=\pi+\arctan\left(-1+\tan \theta_0\right)$ for $\theta_0 \in [0, \pi/4)\cup (\pi/2, \pi)$, or $\theta_1=\arctan\left(-1+\tan \theta_0\right)$ for $\theta_0 \in [\pi/4, \pi/2)$, for $0\le x \le y \le 1$, the resolvent is such that
\begin{align*} 
\( 2 \sqrt{z} \cos \sqrt{z} \cos^2\theta_0 + \sin \sqrt{z} \left(-1-z - (1-z) \cos 2 \theta_0+ z \sin 2 \theta_0\right) \)	G(z,x,y)
\end{align*}
is equal to 
\begin{align*}
& 2 \sin \theta_0 \left( \cos \theta_0 \sin(\sqrt{z} (1-y)) +\sqrt{z} \cos (\sqrt{z} (1-y)) \left(\sin \theta_0-\cos \theta_0 \right) \right)  \, \cos (\sqrt{z} x)
\\
& \hspace{1em}  + 2 \cos \theta_0 \left( - \cos \theta_0 \sin(\sqrt{z} (1-y)) +\sqrt{z} \cos (\sqrt{z} (1-y)) \left(\cos \theta_0-\sin \theta_0 \right) \right)  \,\frac{\sin (\sqrt{z} x)}{\sqrt{z} }. 
\end{align*}
Calculating $\lim_{z\to 0^{+}} z^n \, G(z, x,y)$ for $n=1, 2,\ldots$ give a Riesz projection with a kernel
\begin{equation*}
    p(x,y)=\frac{6 \left(x \cos \theta_0 - \sin \theta_0\right) \, \left(y \cos \theta_0 - \sin \theta_0\right)}{-4 + 2 \cos 2 \theta_0+ 3 \sin 2 \theta_0}
\end{equation*}
and a nilpotent operator $D\equiv 0$. Note by simple calculus that the expression $-4 + 2 \cos 2 \theta_0+ 3 \sin 2 \theta_0$ is never zero, since it has a maximum value of $-4+\sqrt{13}\approx -0.394$ at 
$\theta_0=\frac{1}{2}\arctan \frac{3}{2}$ and minimum value of $-4+\sqrt{13}\approx -7.606$ at $\theta_0=\frac{\pi}{2} + \frac{1}{2}\arctan \frac{3}{2}$. 
The Green's function is such that 
\begin{align*} 
-20 \left(-4 + 2 \cos 2 \theta_0+ 3 \sin 2 \theta_0\right)^2 \,	G_0(x,y)
\end{align*}
is equal to 
\begin{align*} \label{eq:resolvent-GSARS}
a_0(x,y)+ a_2(x,y) \cos 2 \theta_0 + a_4(x,y) \cos 4 \theta_0 + b_2(x,y) \sin 2 \theta_0 + b_4(x,y) \sin 4 \theta_0,  
\end{align*}
where 
\begin{align*}
a_0(x,y)&=-132-90x + 360 y -150 x^2 + 204 x y -150 y^2 - 15 x^3 - 45 x^2 y - 45 x y^2 - 15 y^3 - 30 x^3 y - 30 x y^3\\
a_2(x,y)&= 120 -40 x- 360 y+180 x^2+72 x y +180 y^2- 20 x^3 y - 20 x y^3\\
a_4(x,y)&= 12+ 50 x - 30 x^2-132 x y - 30 y^2 + 15 x^3 + 45 x^2 y + 45 x y^2 + 15 y^3+ 10 x^3 y + 10 x y^3\\
b_2(x,y)&= 140 + 144 x - 336 y+ 90 x^2 - 360 x y + 90 y^2+ 40 x^3+120 x^2 y + 120 x y^2+ 40 y^3+ 30 x^3 y + 30 x y^3\\
b_4(x,y)&= -50 +12 x +132 y - 45 x^2 - 45 y^2 - 10 x^3 - 30 x^2 y- 30 x y^2- 10 y^3+ 15 x^3 y+ 15 x y^3\\. 
\end{align*}
With $T(x,y)= \kappa(x,y)- G_0(x,y)$, the expression of the integral operator $\mathcal{T} f (x) $ is a cubic polynomial, hence the eigenvector takes the form 
$f(x)=a_0+a_1 x+a_2 x^2+ a_3 x^3$ and the corresponding $4\times 4$ matrix $M_{\mathcal{T}}= \left(m_{ij}\right)$ is such that
\begin{align*}
-20 \left(-4 + 2 \cos 2 \theta_0+ 3 \sin 2 \theta_0\right)^2 \,	m_{11}&= \frac{473}{4} - 80 \cos 2 \theta_0- \frac{73}{4} \cos 4 \theta_0 - 132 \sin 2 \theta_0+ \frac{63}{2} \sin 4 \theta_0 
\\
-20 \left(-4 + 2 \cos 2 \theta_0+ 3 \sin 2 \theta_0\right)^2 \,	m_{12}&= \frac{123}{4} - \frac{115}{3} \cos 2 \theta_0- \frac{59}{6} \cos 4 \theta_0 - \frac{137}{2} \sin 2 \theta_0+ \frac{57}{4} \sin 4 \theta_0
\\
-20 \left(-4 + 2 \cos 2 \theta_0+ 3 \sin 2 \theta_0\right)^2 \,	m_{13}&= \frac{171}{4} - 26 \cos 2 \theta_0- \frac{27}{4} \cos 4 \theta_0 - \frac{142}{3} \sin 2 \theta_0+ \frac{28}{3} \sin 4 \theta_0
\\
-20 \left(-4 + 2 \cos 2 \theta_0+ 3 \sin 2 \theta_0\right)^2 \,	m_{14}&= \frac{232}{7} - 20 \cos 2 \theta_0- \frac{36}{7} \cos 4 \theta_0 - \frac{1278}{35} \sin 2 \theta_0+ \frac{246}{35} \sin 4 \theta_0\\
-20 \left(-4 + 2 \cos 2 \theta_0+ 3 \sin 2 \theta_0\right)^2 \,	m_{21}&= -\frac{429}{2} +169 \cos 2 \theta_0+ \frac{47}{2} \cos 4 \theta_0 + \frac{457}{2} \sin 2 \theta_0- \frac{263}{4} \sin 4 \theta_0 \\
-20 \left(-4 + 2 \cos 2 \theta_0+ 3 \sin 2 \theta_0\right)^2 \,	m_{22}&= -\frac{473}{4} +80 \cos 2 \theta_0+ \frac{73}{4} \cos 4 \theta_0 +132 \sin 2 \theta_0- \frac{63}{2} \sin 4 \theta_0 \\
-20 \left(-4 + 2 \cos 2 \theta_0+ 3 \sin 2 \theta_0\right)^2 \,	m_{23}&= -82 +52 \cos 2 \theta_0+ 14\cos 4 \theta_0 +93 \sin 2 \theta_0- \frac{41}{2} \sin 4 \theta_0 \\
-20 \left(-4 + 2 \cos 2 \theta_0+ 3 \sin 2 \theta_0\right)^2 \,	m_{24}&= - \frac{8787}{140} + \frac{1346}{35} \cos 2 \theta_0+ \frac{1571}{140} \cos 4 \theta_0 + \frac{502}{7} \sin 2 \theta_0- \frac{106}{7} \sin 4 \theta_0 \\
-20 \left(-4 + 2 \cos 2 \theta_0+ 3 \sin 2 \theta_0\right)^2 \,	m_{31}&= \frac{345}{2} - 180 \cos 2 \theta_0+ \frac{15}{2} \cos 4 \theta_0 -150 \sin 2 \theta_0+ 60 \sin 4 \theta_0 \\
-20 \left(-4 + 2 \cos 2 \theta_0+ 3 \sin 2 \theta_0\right)^2 \,	m_{32}&= 90 -90 \cos 2 \theta_0- 85 \sin 2 \theta_0+ \frac{65}{2} \sin 4 \theta_0 \\
-20 \left(-4 + 2 \cos 2 \theta_0+ 3 \sin 2 \theta_0\right)^2 \,	m_{33}&= \frac{245}{4} - 60 \cos 2 \theta_0- \frac{5}{4} \cos 4 \theta_0 - 60  \sin 2 \theta_0+ \frac{45}{2} \sin 4 \theta_0 \\
-20 \left(-4 + 2 \cos 2 \theta_0+ 3 \sin 2 \theta_0\right)^2 \,	m_{34}&= \frac{93}{2} - 45 \cos 2 \theta_0- \frac{3}{2} \cos 4 \theta_0 - \frac{93}{2} \sin 2 \theta_0+ \frac{69}{4} \sin 4 \theta_0 \\
-20 \left(-4 + 2 \cos 2 \theta_0+ 3 \sin 2 \theta_0\right)^2 \,	m_{41}&= 30 +10 \cos 2 \theta_0- 20 \cos 4 \theta_0 - 55 \sin 2 \theta_0+ \frac{5}{2} \sin 4 \theta_0 \\
-20 \left(-4 + 2 \cos 2 \theta_0+ 3 \sin 2 \theta_0\right)^2 \,	m_{42}&= \frac{35}{2} + \frac{20}{3} \cos 2 \theta_0- \frac{65}{6} \cos 4 \theta_0 - 30 \sin 2 \theta_0 \\
-20 \left(-4 + 2 \cos 2 \theta_0+ 3 \sin 2 \theta_0\right)^2 \,	m_{43}&= \frac{25}{2}+ 5 \cos 2 \theta_0- \frac{15}{2} \cos 4 \theta_0 - \frac{125}{6} \sin 2 \theta_0+ \frac{5}{12} \sin 4 \theta_0 \\
-20 \left(-4 + 2 \cos 2 \theta_0+ 3 \sin 2 \theta_0\right)^2 \,	m_{44}&= \frac{39}{4} +4 \cos 2 \theta_0- \frac{23}{4} \cos 4 \theta_0 - 16 \sin 2 \theta_0- \frac{1}{2} \sin 4 \theta_0. 
\end{align*}
The determinant of $M_{\mathcal{T}}$ is zero and so is the case of the $3\times 3$ subdetermiant extracted from the first three columns (but not the $2\times 2$ subdeterminant of the first two columns). Hence the rank of the perturbation is 2. A calculation on Wolfram Mathematica\textsuperscript{\textregistered} confirms this, and shows that $\lambda=0$ is a double eigenvalue. 

\vskip 1 cm

\subsection*{Acknowledgments}
We would like to thank Fritz Gesztesy for introducing us to some aspects of this problem and for sharing an advanced copy of \cite{GNZ-AMSBook2023}; Jussi Behrndt for reference \cite{BehrndtHassideSnoo2020}; and Mark Ashbaugh for useful discussions. We acknowledge email correspondence with Walter A. Strauss regarding the ``unusual'' eigenvalue problem discussed in Sec.~\ref{sec:KvN Resolvent} and thank him for his feedback. The authors thank the anonymous referee whose comments helped improve the article significantly. 

L.H. thanks and acknowledges N.S.'s NSF grant DMS-1912747 during his visits to the UC Davis campus in 2019 and 2022 where some of this work was done.

N.S. was partially supported by his NSF grants DMS-1912747, CCF-1934568, and his ONR grant N00014-20-1-2381.
\section*{REFERENCES}
%

\end{document}